\documentclass[11pt]{article}
\date{}
\usepackage{latexsym}
\usepackage{amsmath,amsthm,amssymb,enumerate,paralist}
\usepackage{cite}
\usepackage[colorlinks=true, pdfstartview=FitV, 
pagebackref=false]{hyperref}
\numberwithin{equation}{section}
\allowdisplaybreaks

\def\cD{{\mathcal D}}
\renewcommand{\epsilon}{\varepsilon}
\usepackage[margin=1in]{geometry}

\def\tG{\tilde{G}}

\def\hW{\wh{W}}

\def\bd{{\bf d}}
\def\wT{\widehat{T}}
\def\wK{\widehat{K}}
\def\wW{\widehat{W}}

\newcommand{\beq}[1]{\begin{equation}\label{#1}}
\newcommand{\eeq}{\end{equation}}
\newcommand{\bem}[1]{\begin{multline}\label{#1}}
\newcommand{\card}[1]{\left| #1 \right|}
\newcounter{rot}

\def\leb{\leq_b}

\def\a{\alpha} \def\b{\beta} \def\d{\delta} \def\D{\Delta}
\def\e{\epsilon} \def\f{\phi} \def\F{{\Phi}}  \def\g{\gamma}
\def\G{\Gamma}  \def\k{\kappa}
\def\z{\zeta} \def\th{\theta}    \def\l{\lambda}
\def\La{\Lambda} \def\m{\mu} \def\n{\nu} \def\p{\pi}
\def\r{\rho}  \def\s{\sigma} 
\def\t{\tau} \def\om{\omega}

\def\cY{{\cal Y}}

\newtheorem{maintheorem}{Theorem}
\newtheorem{maincoro}[maintheorem]{Corollary}
\newtheorem{theorem}{Theorem}[section]

\newtheorem{lemma}[theorem]{Lemma}

\newtheorem{Remark}[theorem]{Remark}
\newtheorem{definition}[theorem]{Definition}
\newtheorem{conjecture}[theorem]{Conjecture}
\newtheorem*{conjecture*}{Conjecture}
\def\cW{{\cal W}}
\def\cX{{\cal X}}
\def\cG{{\cal G}}

\newcommand{\wh}[1]{\widehat{#1}}

\newcommand{\rdup}[1]{{\left\lceil #1 \right\rceil }}
\newcommand{\rdown}[1]{{\left\lfloor #1 \right\rfloor}}

\newcommand{\brac}[1]{\left(#1\right)}

\newcommand{\bfrac}[2]{\left(\frac{#1}{#2}\right)}

\def\cEm{{\cal E}_{\max}}

\newcommand{\CW}[1]{\cW^{\,#1}}
\newcommand{\set}[1]{\left\{#1\right\}}

\def\E{\mathbb{E}}
\def\Pr{\mathbb{P}}
\def\wPr{\widehat{\mathbb{P}}}

\newcommand{\ignore}[1]{}

\newcommand{\cA}{{\cal A}}

\def\Gd{G_{\textbf{d}}}
\def\Kd{K_{\textbf{d}}}
\def\wF{\widehat{F}}

\def\bd{\textbf{d}}

\newcommand{\GC}{{\mathcal{C}_1}} 
\newcommand{\TC}[1][\mathcal{C}_1]{#1^{(2)}} 
\newcommand{\tTC}{G_F} 
\newcommand{\tKer}{K_F} 
\newcommand{\Tm}{T_\textsc{mix}}
\newcommand{\Tc}{T_\textsc{cov}}

\def\am{\ell_{\max}}
\def\bm{\ell_{\min}}
\def\x{v_{f^*}}
\def\tW{\CW{G_0^*\to V_0}}
\def\ttW{\CW{G_0^*\to V_\l}}
\def\sW{\CW{G_0}}
\def\GW{\CW{G_0^*}}
\def\cWa{\CW{\tTC}}

\def\deg{\text{d}}
\def\bq{{\bf q}}
\def\cQ{{\mathcal Q}}
\def\wS{\widehat{S}}
\def\wG{\widehat{G}}

\def\wV{\widehat{V}}

\def\wE{\widehat{E}}
\def\cN{{\cal N}}
\begin{document}
\makeatletter
\title{Cover time of a random graph with a degree sequence II: \\
Allowing vertices of degree two}
\author{Colin Cooper\thanks{Department of  Computer Science, King's College,
University of London, London WC2R 2LS, UK.
Email: {\tt colin.cooper@kcl.ac.uk}}\and Alan
Frieze\thanks{Department of Mathematical Sciences, Carnegie Mellon
University, Pittsburgh PA15213, USA.  Email: {\tt alan@random.math.cmu.edu}.
Research supported in part by
NSF Grant DMS-0502793.}
\and Eyal Lubetzky\thanks{Microsoft Research,
One Microsoft Way, Redmond, WA 98052, USA. Email:
{\tt eyal@microsoft.com}}}
 \maketitle \makeatother
 \vspace{-0.5cm}
\begin{abstract}
We study the cover time of a random graph
chosen uniformly at random from the set of graphs with vertex set
$[n]$ and degree sequence $\textbf{d}=(d_i)_{i=1}^n$.
In a previous work \cite{ACF}, the asymptotic cover time was obtained
under a number of assumptions on $\textbf{d}$, the most
significant being that $d_i\geq 3$ for all $i$.
Here we replace this assumption by $d_i\geq 2$. As a corollary, we establish the asymptotic cover time for the 2-core of the emerging giant component of $\cG(n,p)$.
\end{abstract}

\section{Introduction}\label{intro}
Let $G=(V,E)$ be a connected graph with $n$ vertices and $m$ edges.
For $v\in V$, let $C_v$ be the expected time for a simple random
walk $\cW_v$ on $G$ starting at $v$, to visit every vertex of $G$. The
{\em (vertex) cover time} $\Tc(G)$ of $G$ is defined as $\Tc(G)=\max_{v\in
V}C_v$. 
It is a classic result of Aleliunas, Karp,
Lipton, Lov\'asz and Rackoff~\cite{AKLLR} that $\Tc(G) \le 2m(n-1)$. Feige~\cite{Feige1,Feige2}
showed that the cover time of any connected graph $G$ satisfies
$(1-o(1))n\ln n\leq \Tc(G)\leq (1+o(1))\frac{4}{27}n^3.$
Between these two extremes, the cover time, both
exact and asymptotic,
has been extensively studied for different classes of graphs (see, e.g.,~\cite{Aldous} for an introduction to the topic).


In the context of random graphs, a basic question is to understand the cover time for the giant component $\GC$ of the celebrated Erd\H{o}s-R\'enyi~\cite{ER} random graph model $\cG(n,p)$. Decomposing the giant $\GC$ into the \emph{2-core} $\TC$ (its maximal subgraph of minimum degree 2) and collection of trees decorating $\TC$, much is known about their structure (see, e.g., the characterization theorems in the recent works~\cite{Lub,Lub2}). However, our understanding of the cover time for these remains incomplete.

It is well-known that for $G\sim\cG(n,p=c/n)$ with $c>1$ fixed,
the giant component $\GC$ is roughly of size $x n$ where $x=x(c)$ is the
solution in $(0,1)$ of $x=1-e^{-cx}$. Cooper and Frieze~\cite{CFgiant} showed that in this regime
\begin{align}
\Tc(\GC)\sim  \frac{cx(2-x)}{4(cx-\ln c)} n \ln^2 n
\quad \mbox{ and }\quad
\Tc(\TC)\sim \frac{cx^2}{16(cx-\ln c)} n \ln^2 n
\label{G1Ca-c}
\end{align}
with high probability (w.h.p.), i.e., with probability tending to $1$ as $n\to\infty$. However, analogous results
for $p=(1+\epsilon)/n$ with $\epsilon=o(1)$, $\epsilon^3 n \to\infty$ (the \emph{emerging} giant component) were unavailable.

Barlow \emph{et al.}~\cite{BDNP} showed that when $p=(1+\epsilon)/n$ with $n^{-1/3}\ll \epsilon \ll 1$ (here and in what follows
we let $A_N\ll B_N$ denote $\lim_{N\to\infty}A_N/B_N=0$) the cover time $\Tc(\GC)$ is of order $n \log^2(\e^3 n)$.
With this in mind, substituting $c=1+\epsilon$ with $\epsilon>0$ in the estimates of~\eqref{G1Ca-c}, and noting that the aforementioned $x(c)$ becomes $2\e+O(\e^2)$, shows that for small \emph{fixed} $\epsilon>0$, w.h.p.
\begin{align}
\Tc(\GC)=(1+O(\e))n\ln^2(\epsilon^3n)
\quad \mbox{ and }\quad
\Tc(\TC) = \frac{\e+O(\e^2)}{4}n\ln^2(\epsilon^3n)\,,
\label{G1Ca-eps}
\end{align}
and one may expect these results to hold throughout the emerging giant regime of $n^{-1/3} \ll \epsilon \ll 1$.

A natural step towards this goal is to exploit the well-known characterizations of $\GC$, its 2-core and its \emph{kernel}: as mentioned above, by stripping the giant component of its attached trees one arrives at the 2-core $\TC$. By further shrinking every induced path in $\TC$ into a single edge one arrives at the kernel $K$ (see \S\ref{config} for more details). It was shown by {\L}uczak~\cite{Luczak} that the kernel of the emerging giant component is a random multi-graph on a certain degree sequence, and so, potentially, the cover times of $K$, $\TC$ and $\GC$ could all be determined as a consequence of general results on the cover-time of random graphs with a given degree sequence.

Promising in that regard is a framework developed by Cooper and Frieze, which was already successful in tackling this problem for a variety of random graph models, notably including random regular graphs~\cite{CFC} and random graphs with certain degree sequences~\cite{ACF} (also see~\cite{CFC,CFgiant,CFreg,CFweb,Dnp}). However, among the various conditions on the degree sequence in~\cite{ACF}, a main caveat was the requirement that the minimal degree should be at least $3$, rendering this machinery useless for analyzing the 2-core.

In this paper we eliminate this restriction and
allow vertices of degree 2 in the degree sequence.
Of course, if our degree sequence $\bd$ features linearly many degrees that are 2 --- as in the case of the 2-core of the emerging giant --- a
uniformly chosen graph with these degrees will typically contain linearly many isolated cycles, which would have to be removed.
To avoid this issue, we let the degree 2 vertices arise as they do in the giant component, as subdivision of kernel edges:
\begin{compactitem}[$\bullet$]
  \item Given $\bd=(d_1\leq d_2\leq \cdots\leq d_n)$ with $d_i \geq 2$ for all $i$, let $\nu_2$ be the number of degree 2 vertices, and let $\bd_3$ be the degree sequence restricted to all $i$ such that $d_i \geq 3$.
\item Choose the kernel $\Kd \sim {\cal G}_{\textbf{d}_3}$, i.e., uniformly from all multi-graphs with degree sequence $\textbf{d}_3$.
\item Replace each edge $e$ of $\Kd$ by a path $P_e$ of length $\ell_e$ (edges), where the values of $\{\ell_e : e\in E(\Kd)\}$ are uniform
over all $\binom{\n_2+|E(\Kd)|-1}{\nu_2}$ possible choices, to obtain the final graph $\Gd$.
\end{compactitem}
Under several natural conditions on $\bd$ (e.g., satisfied when it has a power law/exponential tail, as in the 2-core of $\GC$), detailed next, we can determine the asymptotic cover time of $\Tc(\Gd)$.

\begin{definition}\label{def-nice}
Let $\bd=(d_i)_{i=1}^n$ and let $\nu_j = \#\{i : d_i=j\}$ count the degree-$j$ vertices in $\bd$.
Let $N,M,d$ be the number of vertices, number of edges and minimum degree in the associated kernel:
\[ \mbox{$N = \sum_{j \geq 3} \nu_j\,,\qquad M = \frac12\sum_{j\geq 3} j \nu_j\,\,\qquad d = \min\{ j\geq 3 : \nu_j \neq 0\}\,.$} \]
We say that \textbf{d} is \emph{nice} (and similarly, ${\cal G}_{\bd}$ is nice) if it satisfies the following conditions:
\begin{align}
\mbox{$N\to\infty$ as $n\to\infty$} & \qquad\mbox{(diverging kernel)}\,, \label{nice-a-kernel}\\
2\leq d_1\leq d_2\leq \cdots \leq d_n\leq N^{\z_0}\text{ where }\z_0= o(1)&\qquad\mbox{(sub-poly degrees)}\,,\label{zeta0}\\
\mbox{$\sum_{j\geq 3}j^3 \n_j \leq a_0M$ for an absolute constant $a_0\geq 1$}&\qquad\mbox{(3rd moment bound)}\,,\label{nice-c-kernel}\\
\n_d\geq \a N \mbox{ for an absolute constant $\a>0$}&\qquad\mbox{(minimum kernel degree)}
\,.\label{nice-d-kernel}
\end{align}
\end{definition}
Observe that without condition~\eqref{nice-a-kernel}, the graph $\Gd$ would be disconnected w.h.p.
The upper bound in~\eqref{zeta0} is for convenience, and we can assume without loss of generality that
\beq{zetalarge}
\z_0\gg \frac{\ln\ln N}{\ln N}.
\eeq
Condition~\eqref{nice-c-kernel} allows us to work directly with the
configuration model of Bollob\'as~\cite{Boll1}. It does, however, restrict our attention to cases where
the average degree in the kernel (thus overall) is bounded, as
Jensen's inequality implies that $\sum_{j\geq 3}j^3 \nu_j \geq N(2M/N)^3$ and so
\beq{JJ}
\frac{2M}{N}\leq \bfrac{a_0}{2\sqrt{2}}^{1/2}\leq a_0.
\eeq
Finally, the minimum kernel degree $d$ (the focus of~\eqref{nice-d-kernel}) will be featured in the
statement of our main theorem.
We note that some of the assumptions above can be relaxed at the cost of
some extra technicalities that would detract
from the main new ideas of the paper.

The following two important classes of degree sequence are nice:
\begin{enumerate}[(i)]
\item Exponential tail: there exist real non-negative
constants $\a,\b$ with $\b<1$ and a positive integer $j_0\ge 3$ such that
$\n_j/N \leq \a \b^j$ for $j\ge j_0$.
\item Power law (moderate): there exist real positive
constants $c,\g$ with $\g\geq 3$ and a positive
integer $j_0\ge 3$ such that
$\n_j/N \leq c j^{-\g}$ for $j\ge j_0$, and the maximum degree is $N^{o(1)}$.
\end{enumerate}
This of course includes degree sequences with bounded maximum degree $\D_0$.

The main result of this paper is the following.
\begin{maintheorem}\label{th1}
Let $\bd$ be a nice degree sequence as per Definition~\ref{def-nice}. The following hold w.h.p.
\begin{compactenum}
  [(a)]
  \item If $\nu_2 = M^{o(1)}$ then
  \begin{align*}
\Tc(\Gd)\sim \frac{2(d-1)}{d(d-2)}M\ln M\,.
  \end{align*}
 \item If $\n_2=M^\a$ for some fixed $0<\a<1$ then
 \begin{align*}
\Tc(\Gd)\sim \max\left\{ \frac{2(d-1)}{d(d-2)}\,, \phi_{\a,d}\right\} M\ln M\,,
 \end{align*}
 where
\[\f_{\a,d}=\min\set{\t:
\min_{k=1,2.\ldots}\set{(1-\a)k+\frac{\t}{2}
\brac{\frac{1}{\rdown{(k+1)/2}+\frac{1}{d-2}}+
\frac{1}{\rdup{(k+1)/2}+\frac{1}{d-2}}}}\geq 1}\,.\]
 \item If $\nu_2 =\Omega(M^{1-o(1)})$ then
 \begin{align*}
   \Tc(\Gd)\sim \frac{m\ln^2M}{-8\ln(1-\xi)}\,,
 \end{align*}
where $m = |E(\Gd)| = \nu_2+M$ and
 \beq{xid}
\xi= M / m\,.
\eeq
\end{compactenum}
\end{maintheorem}
Note that as $\a\to 1$ we will have $\f_{\a,d}\sim \frac{1}{8(1-\a)}$
and $-\ln(1-\xi)\sim (1-\a)\ln M$.
So, as $\a\to1$ we see that Cases~(b) and~(c) are consistent.
Finally, observe that the condition in Case~(c) can also be written
as $-\ln(1-\xi)=o(\ln M)$.

Going back to the cover time of $\TC$, the 2-core of $\GC$, we see immediately that
the estimate of~\cite{CFgiant} on its cover time (see~\eqref{G1Ca-c}) readily follows from Case~(c) of Theorem~\ref{th1}, whence
\begin{align*}
  \n_2\sim c^2x^2e^{-cx}n/2\quad\mbox{ and }\quad M\sim cx^2(1-ce^{-cx})n/2.
\end{align*}
Furthermore, Theorem~\ref{th1} implies
that the estimate for $\Tc(\TC)$ in case $p=(1+\epsilon)/n$ with $\epsilon>0$ fixed (see~\eqref{G1Ca-eps}) extends to the entire emerging supercritical regime. Indeed, by known characterizations of the 2-core (see, e.g.,~\cite{Lub}) this case corresponds to $M\sim 2\e^3n$ and $\n_2\sim 2\e^2n$.
\begin{maincoro}\label{CF2}
Let $p=(1+\e)/n$ where $\e=o(1)$ and $\e^3 n \to\infty$.
Then w.h.p.,
$$\Tc(\TC)\sim \frac{\e}{4}n\ln^2(\e^3n).$$
\end{maincoro}

We conclude with an open problem.
While this work eliminated the restrictive assumption of minimum degree 3 for the degree sequence under consideration, vertices of degree 1 still pose a significant barrier in the analysis. It would be interesting to extend Theorem~\ref{th1} to degree sequences that do include a linear number of such vertices, towards establishing the following conjecture for the cover time of the emerging giant component.
\begin{conjecture*}Let $p=(1+\e)/n$ where $\e=o(1)$ and $\e^3 n \to\infty$.
Then w.h.p.,
$$\Tc(\GC)\sim n\ln^2(\e^3n).$$
\end{conjecture*}

\subsection*{Outline of the paper}
We begin with those arguments that are common to all parts of Theorem~\ref{th1}.
Section \ref{config} describes the configuration
model of graphs with a fixed degree
sequence that we will use throughout.
Section \ref{degtwo} describes the distribution
of the number of vertices ($\ell_e-1$)
that are placed on each edge $e$ of the kernel.
Section \ref{treelike} shows that most
vertices have tree like neighbourhoods. Rapid
mixing is an important property of our
graphs and Section \ref{conductance1} gives an initial analysis of conductance.

Lemma~\ref{MainLemma} is our main tool in proving an upper bound on cover time.
Let $T$ be a ``mixing time''.
Fix a vertex $v$ and let $\p_v$ denote the
steady state probability that a random walk
on a graph $G$ is at $v$.
Let $R_v$ be the expected number of
returns to $v$ of a random walk, started at $v$,
within time $T$. Broadly
speaking, Lemma~\ref{MainLemma} says that if we define the event
\beq{aCv}
\cA_t(v)=\set{\text{vertex $v$ is not visited
by the walk during the interval $[T,t]$}}
\eeq
then, if $T\p_v=o(1)$ and another more technical condition holds,
then to all intents and purposes,
$$\Pr(\cA_t(v))\approx e^{-t\p_v/R_v}.$$
The above inequality has been used to prove an
upper bound in \cite{ACF,CFreg,CFweb,CFgiant,Dnp,CFgeo} and several other papers.
In this paper we use it in inequality \eqref{shed} below.

\begin{itemize}
\item {\bf The case where $\n_2$ is not too large:} \label{ntl}
We begin the proof of Case~(c) of Theorem~\ref{th1} in Section~\ref{casec1}, where we consider the case of $\n_2$ ``close'' to $M$; this will be Case (c1). In this range, $\xi$ is not too small and Lemma~\ref{MainLemma} is sufficient
to the task.
We have $T=O(\ln^{O(1)}M/\xi^2)$ and
$\p_v=O(\ln M/(\xi M))$ and $T\p_v=o(1)$.
Section \ref{conditions0} proves this and
verifies the more technical condition.
So, Lemma~\ref{MainLemma} can be applied
directly in this case. Given this, the main task that
arises is
in estimating the values, $R_v$. The number of
returns to $v$ is related in a strong way to the
electrical resistance
of its ``local neighbourhood''. This reduces to
estimating the resistance $R(T)$ of a bounded
depth binary tree $T$ where the
resistance of an edge is equal to a geometric random
variable with success probability $\xi$.
This is the content
of Section \ref{Rv0}. We only prove bounds on
the probability that $R(T)$ is large.

\item {\bf The case where $\n_2$ is large:} \label{n2L}
Section \ref{casec2} deals with the case where
$\n_2$ is large with respect to $M$; we split this into Case (c2) where $\n_2$ is large but not ``too large'' and Case (c3) where $\n_2$ is very large. We will see that Case (c2) takes up most of our time and that Case (c3) can easily be reduced to the former case. We immediately
run into a problem in using Lemma~\ref{MainLemma}. As $\n_2$ grows, the mixing time
of a walk grows like $(\n_2/M)^2$ and the steady state
values decrease like
$1/(\n_2M)$. This means that for $\n_2$ large,
$T\p_v\gg1$. This is where we
need some new ideas. We choose some $\om=N^{o(1)}$
and define $\ell^*=1/{\xi\om}$. A typical edge $e$
of the kernel will give rise to a path $P_e$
of length $\ell_e=\Theta(1/\xi)$.
We divide $P_e$ into $\Theta(\om)$ sub-paths
of length $\ell\in [\ell^*,2\ell^*]$.
(Because $\ell^*$ does not necessarily divide $\ell_e$,
the value of $\ell$ may vary from sub-path to sub-path).
We then replace these sub-paths by edges of
weight $\ell^*/\ell$ to create an edge-weighted
graph $G_0$.
We consider a random walk $\cW_0$
where at a vertex, we choose the next edge to
cross with probability proportional to weight.
We argue that the {\em edge} cover time of
$\cW_0$ is approximately $(\ell^*)^2$ times the
cover time we are interested in.

At first glance, this should eliminate the
$T\p_v\to\infty$ problem, as $T$ should be
$O(\ln^{O(1)}M/\om^2)$ and so $\p_v=O(\ln M/(\om M)$.
Unfortunately, this bound on $T$ is false:
the problem comes from edges of the kernel
for which $\ell_e<\ell^*$. These edges give
rise to single edges of weight
$\ell^*/\ell_e$ in $G_0$. In the worst-case
we have $\ell_e=1$ and we have an edge $f=(w_1,w_2)$
of weight
$\ell^*$. The walk $\cW_0$ could spend a lot
of time travelling back and forth from $w_1$
to $w_2$ and vice-versa. In any case, such
an edge can reduce the conductance of the
walk $\cW_0$ to $O(1/(\ell^*)^2)$ undoing all
of our work. Our solution to this is
to modify the walk so that it ``races along''
edges of high weight. This will give
us a walk that satisfies the conditions of the
lemma. We then have to bound the time we ignored,
to which end we apply a concentration
inequality of Gillman \cite{Gill}.

Section \ref{struct} deals with structural properties
associated with this case. In
particular showing that there are relatively
few vertices of high weight. It also deals in some
detail with properties that are
needed for
estimates of the conductance of our modified walk.
Section \ref{surro} deals in detail
as to how we make edges out of sub-paths.
The goal from now on is to estimate $\Pr(\cA_t(f))$
where $f$ is some edge of $G_0$. We
deal with each $f$ separately in the
sense that we create a graph $G$ for each $f$.
Splitting $f$ by adding a vertex $v_f$ to
its middle. Then visiting $v_f$
will be equivalent to crossing $f$.
Section \ref{complete} uses Gillman's theorem to show that we have not
ignored too many steps.
\end{itemize}
The remainder of the paper is organized as follows.
Sections~\ref{caseb} and~\ref{casea} deal with
Cases~(b) and (a) of Theorem~\ref{th1}.
They are easier to prove than Case~(c),
being closer in spirit to earlier papers.

Section~\ref{lowerbounds} deals with matching
lower bounds on the cover time.
Section~\ref{lowerc} uses the Matthews
bound, see for example~\cite{LPW}. Section~\ref{lowerb}
and Section~\ref{lowera} follow a pattern established in the earlier
mentioned papers. We choose a time $t$ that is a little bit less
than our estimated cover time. We identify a set of vertices $S$
that have not been visited up to time $t$. The size
of $S$ is large in expectation and Chebyshev inequality
combined with Lemma~\ref{MainLemma} to show that $S\neq \emptyset$
w.h.p.

\section{Structural properties}

Recall that for a degree sequence $\bd=(d_1\leq \ldots\leq d_n)$ we let $\nu_j$ count the number of vertices of degree $j$. It will be useful to further define $ V_j = \{i \in V: d_i = j\}$ (so that $\n_j = |V_j|$) as well as
\[ D_k=\sum_{j\geq 3}j^k \n_j \]
(so that $N = D_0$ and $M = D_1 / 2$ are the number of vertices and edges in the kernel, respectively).

\subsection{Configuration model}\label{config}
We make our calculations in the configuration model, see Bollob\'as \cite{Boll1}.
Let $W=[2m]$ be our set
of {\em configuration points} and let $W_i=[d_1+\cdots+d_{i-1}+1,d_1+\cdots+d_{i}]$,
$i\in [n]$, partition $W$. The function $\f:W\to[n]$ is defined by
$w\in W_{\f(w)}$. Given a
pairing $F$ (i.e., a partition of $W$ into $m$ pairs) we obtain a
(multi-)graph $G_F$ with vertex set $[n]$ and an edge $(\f(u),\f(v))$ for each
$\{u,v\}\in F$. Choosing a pairing $F$ uniformly at random from
among all possible pairings $\Omega_W$ of the points of $W$ produces a random
(multi-)graph $G_F$. Let
\beq{F2m}
\mathcal{F}(2m)=\frac{(2m)!}{m!2^m}.
\eeq
This is the number of pairings $F$ of
the points in $W$.

The {\em kernel} $K_F$ is obtained from $G_F$ by repeatedly
replacing {\em induced} paths of
length two by edges. The number of vertices in the kernel
is $N$, the number of vertices of degree at least three
and the number of edges in the kernel is
 $M\leq D_3/2\leq a_0N/2$ by \eqref{nice-c-kernel}.

Let
$$\s=\frac{1}{2m}\sum_{j=1}^nd_j(d_j-1)\leq \frac{2\n_2+D_2}{2\n_2+2M}=O(1)
$$
by Assumption (c).

Assuming that $d_n=o(m^{1/3})$ (as it will be for
nice sequences), the probability that $G_F$
 is simple (no loops or multiple edges) is given by
\begin{equation}\label{psimple}
P_S=\Pr( G_F \text{ is simple}) \sim e^{-\s/2-\s^2/4}=\Omega(1).
\end{equation}
See e.g. \cite{McW}.
Furthermore each simple graph $G \in \cG_{\textbf{d}}$
is equiprobable. We can therefore use
$G_F$ as a replacement model for $\Gd$ in the sense
that any event that occurs w.h.p.\ in $G_F$
will occur w.h.p.\ in $\Gd$.

We argue next that:
\begin{lemma}\label{cl1}
The distribution of $K_F$ is that of a configuration
model where $W$ is replaced by $\hW=W_{\n_2+1}\cup W_{\n_2+2}\cup\cdots\cup W_n$.
\end{lemma}
\begin{proof}[\textbf{\emph{Proof}}]
Indeed, we can define a map
$\psi:\Omega_W\to \Omega_{\hW}$ such that for all $F_1,F_2\in \Omega_{\hW}$
we have $|\psi^{-1}(F_1)|=|\psi^{-1}(F_2)|$.
Each induced path $P$ of $G_F$ comes from a set of
pairs $e_i=\set{x_i,y_i},\,i=1,2,\ldots,r$ where
(i) $\f(x_1),\f(y_r)\notin V_2$ (= the set of vertices
of degree two) and (ii) $\f(z)\in V_2$ for
$z\in \set{x_2,\ldots,x_r,y_1,\ldots,y_{r-1}}$.
Replacing $e_i,\,i=1,2,\ldots,r$ by $\set{x_1,y_r}$
defines $\psi(F)\in \Omega_{\hW}$.
The number of $F\in \Omega_W$
that map onto a fixed $F'\in \Omega_{\hW}$ depends
only on $\n_2,m$ and $N$. This implies the lemma.
\end{proof}

\subsection{Distribution of vertices of degree two}\label{degtwo}
We can therefore obtain $F\in \Omega_W$ by first
randomly choosing $F'\in \Omega_{\hW}$ and then
replacing each edge
$e$ of $G_{F'}$ by a path $P_e$. The next thing to
tackle is the distribution of the lengths of
these paths.
Let $\ell_e$ be the length of the path $P_e$.
Suppose now that the edges of $F'$ are $e_1,e_2,\ldots,e_M$
and write $\ell_j$ for $\ell_{e_j}$.
\begin{lemma}\label{cl2}
The vector $(\ell_1,\ell_2,\ldots,\ell_M)$ is chosen uniformly from
$$\set{\ell_i\geq 1, i=1,2,\ldots,M\text{ and }
\ell_1+\ell_2+\cdots+\ell_M=\n_2+M}.$$
\end{lemma}
\begin{proof}[\textbf{\emph{Proof}}]
Each such vector arises in $\n_2!$ ways. Indeed, we order $V_2$ and then
assign the associated vertices
in order, $\ell_1-1$ to
$e_1$ to create $P_{e_1}$, $\ell_2-1$ to
$e_2$ to create $P_{e_2}$ and so on.
\end{proof}

Some calculations can be made simpler if we observe
the alternative description of the
distribution of
$(\ell_1,\ell_2,\ldots,\ell_M)$.
\begin{lemma}\label{lem1}
Let $Z$ be a geometric random variable with success
probability $\xi$. ($\xi$ can be any value
between 0 and 1) here).
Then $(\ell_1,\ell_2,\ldots,\ell_M)$ is distributed
as $Z_1,Z_2,\ldots,Z_M$ subject to $Z_1+Z_2+\cdots+Z_M=\n_2+M$,
where $Z_1,Z_2,\ldots,Z_M$
are independent copies of $Z$.
\end{lemma}
\begin{proof}[\textbf{\emph{Proof}}]
\begin{align*}
&\Pr((Z_1,Z_2,\ldots,Z_M)=(x_1,x_2,\ldots,x_M)\mid Z_1+Z_2+\cdots+Z_M=\n_2+M)\\
&=\frac{\prod_{i=1}^M(1-\xi)^{x_i-1}\xi}{\sum_{y_1+y_2+\cdots+y_M=\n_2+M}
\prod_{i=1}^M(1-\xi)^{y_i-1}\xi}\\
&=\frac{(1-\xi)^{\n_2}\xi^M}{\binom{M+\n_2-1}{M-1}(1-\xi)^{\n_2}\xi^M}\\
&=\frac{1}{\binom{M+\n_2-1}{M-1}}.\qedhere
\end{align*}
\end{proof}

The best choice for $\xi$ will be that for which
$\E(Z_1+Z_2+\cdots+Z_M)=\n_2+M$, i.e.
$M\xi^{-1}=\n_2+M$. We
therefore take $\xi$ as in \eqref{xid}.

Pursuing this line, let $\wPr$ refer to probabilities of events involving
$Z_1,Z_2,\ldots,Z_M$ {\em without} the conditioning $Z_1+Z_2+\cdots+Z_M=\n_2+M$. (Although $\Pr$ and $\wPr$ refer to the same probability space, this will have some notational conveneience later).
\begin{lemma}\label{lem2}
Let $\xi=\frac{M}{M+\n_2}$ and $M,\n_2\to\infty$.
\begin{enumerate}[(a)]
\item
Let $\z=z_1+z_2+\cdots+z_k$ and $k=o(M)$ where $k\z=o(M+\n_2)$,
\begin{multline*}
\Pr(Z_1=z_1,Z_2=z_2,\cdots,Z_k=z_k\mid Z_1+Z_2+\cdots+Z_M=\n_2+M)\leq\\
\wPr(Z_1=z_1,Z_2=z_2,\cdots,Z_k=z_k)(1+\e)
=\xi^k(1-\xi)^{\z-k}(1+\e),
\end{multline*}
where
\beq{epsdef}
\e=\frac{3k\z}{\n_2+M}.
\eeq
\item
If $k\in\set{1,2}$ and $\z=z_1+\cdots+z_k=o(\n_2)$ then
$$\Pr(Z_i=z_i,\,i=1,\ldots,k\mid Z_1+Z_2+\cdots+Z_M=\n_2+M)=
\xi^k(1-\xi)^{\z-k}(1+\eta)$$
where
$$1+\eta=\brac{1+O\bfrac{\z^2M}{\n_2(\n_2+M)}
+ O\bfrac{\z}{\n_2+M}}.$$
\item
Let $\am=\frac{4(M+\n_2)\ln M}{M}=4\xi^{-1}\ln M$. Then
$$\Pr(\exists e:\;\ell_e\geq \am)=o(1).$$
\item Let $\bm=\rdup{\frac{M+\n_2}{M^2\ln M}}=\rdup{\frac{1}{\xi M\ln M}}$ and
suppose that $\n_2/M\ln M\to\infty$ then
$$\Pr(\exists e:\;\ell_e<\bm)=o(1).$$
\end{enumerate}
\end{lemma}
\begin{proof}[\textbf{\emph{Proof}}]
(a)
Observe that
\begin{align}
&\Pr(Z_1=z_1,Z_2=z_2,\cdots,Z_k=z_k\mid Z_1+ Z_{k+2}+\cdots+Z_M=\n_2+M)\nonumber\\
&=\frac{\Pr((Z_1=z_1,Z_2=z_2,\cdots,Z_k=z_k)\wedge
(Z_{k+1}+Z_2+\cdots+Z_M=\n_2+M-\z)}{\Pr(Z_1+ Z_{k+2}+\cdots+Z_M=\n_2+M)}\nonumber\\
&=\frac{\Pr(Z_1=z_1,Z_2=z_2,\cdots,Z_k=z_k)\Pr(Z_{k+1}+ Z_{k+2}+\cdots+Z_M=\n_2+M-\z)}
{\Pr(Z_1+Z_2+\cdots+Z_M=\n_2+M)}
&=\frac{\binom{\n_2+M-\z-1}{M-k-1}}{\binom{\n_2+M-1}{M-1}},\label{nose}
\end{align}
which, since $\z\geq k$, equals
\begin{align}
&\prod_{i=1}^k\frac{M-i}{\n_2+M-i}\times\prod_{i=1}^{\z-k}
\frac{\n_2-i+1}{\n_2+M-k-i}\label{eq2}\\
&\leq \xi^k\prod_{i=1}^{\z-k}
\frac{\n_2-i+1}{\n_2+M-k-i}
= \xi^k(1-\xi)^{\z-k}\prod_{i=1}^{\z-k}
\brac{1+\frac{(k+1)\n_2-(i-1)M}{(\n_2+M-k-i)\n_2}}
\nonumber\\
&\leq \xi^k(1-\xi)^{\z-k}\brac{1+\frac{(1+o(1))(k+1)}{\n_2+M}}^{\z-k}\label{830}\\
&\leq\xi^k(1-\xi)^{\z-k}(1+\e).\nonumber
\end{align}
(b) Going back to \eqref{eq2} with $k=2$ we use
$$\prod_{i=1}^k\frac{M-i}{\n_2+M-i}=\xi^k\brac{1+O\bfrac{1}{\n_2+M}}$$
and
\begin{align*}
&\prod_{i=1}^{\z-k}\frac{\n_2-i+1}{\n_2+M-k-i}\\
&=\frac{\n_2(\n_2-1)\cdots (\n_2-k)}
{(\n_2+M-\z+k)\cdots(\n_2+M-\z+1)(\n_2+M-\z)}\times\prod_{j=k+1}^{\z-k-1}
\frac{\n_2-j}{\n_2+M-j}\\
&=\brac{1+ O\bfrac{\z}{\n_2+M}}\times (1-\xi)^{\z-k}\times \prod_{j=k+1}^{\z-k-1}
\brac{1-\frac{jM}{\n_2(\n_2+M)}+O\bfrac{j^2M}{\n_2(\n_2+M)^2}}\\
&=(1-\xi)^{\z-k}\times \brac{1+O\bfrac{\z^2M}{\n_2(\n_2+M)}+ O\bfrac{\z}{\n_2+M}}.
\end{align*}
(c) It follows from \eqref{nose} with $k=1$ that
\begin{align}
\Pr(\exists e:\;\ell_e\geq \am)&\leq M\sum_{\z=\ell_{\max}}^{\n_2}
\frac{\binom{M+\n_2-\z-1}{M-2}}{\binom{M+\n_2-1}{M-1}}\nonumber\\
&\leq \frac{2M^2}{\n_2}\sum_{\z=\ell_{\max}}^{\n_2}
\brac{1-\frac{\z}{M+\n_2-1}}^{M-2}\nonumber\\
&\leq \frac{2M^2}{\n_2}\sum_{\z=\ell_{\max}}^{\n_2}
\exp\set{-\frac{(M-2)\z}{M+\n_2-1}}\nonumber\\
&\leq \frac{2M^2}{\n_2}\cdot \exp\set{-\frac{(M-2)\am}{M+\n_2-1}}
\frac{1}{1-e^{-(M-2)/(M+\n_2-1)}}\label{lmax}\\
&\leq \frac{2M^2}{\n_2}\cdot \frac{2}{M^4}\cdot \frac{2(M+\n_2)}{M}\nonumber\\
&=o(1).\nonumber
\end{align}
(d) It follows from (a) with $k=1$ and $\z<\bm$ that
\begin{equation*}
\Pr(\exists e:\;\ell_e< \bm)\leq 2M\bm\xi=o(1).\qedhere
\end{equation*}
\end{proof}
\subsection{Tree like vertices}\label{treelike}
Let a vertex $x$ of $\tKer$ be {\em locally tree like}
if its $\tKer$-neighborhood up to depth
\beq{D0}
L_0=\d_0\ln N
\eeq
contains no cycles.

Here
\beq{delta0}
\d_0\gg\z_0\gg\frac{\ln\ln N}{\ln N}
\eeq
where $\z_0$ is as in \eqref{zeta0}.

A vertex of $\tTC$ is locally tree like
if it lies on a path $P_e$ where $e=(v,w)$
and $v,w$ are both locally tree like.
An edge of $\tTC$ is locally tree like if both of its
endpoints are locally tree like.

\begin{lemma}\label{non}
With $L_0$ as defined in \eqref{D0} we have that for the graph $\tKer$:
\begin{description}
\item[(a)] W.h.p.\ there are at most
$N^{10\d_0\ln a_0}$ non locally tree like vertices,  where $a_0$ is as in \eqref{nice-c-kernel}.
\item[(b)] W.h.p.\ there is at most
one cycle contained in the $(2L_0)$-neighborhood of any
vertex.
\end{description}
\end{lemma}
\begin{proof}[\textbf{\emph{Proof}}]
(a) The expected number of
vertices that are within distance $2L_0$ of
a cycle of length at most $2L_0$ in the graph $\tKer$ can be bounded from
above by
\begin{multline}\label{pa}
\sum_{l=0}^{2L_0}\sum_{k=3}^{2L_0}
\sum_{\substack{v_1,\ldots,v_k\\w_1,\ldots,w_l}}\ \deg(v_1)
\prod_{i=1}^k\frac{\deg(v_i)^2}{M}\prod_{j=1}^l\frac{\deg(w_j)^2}{M}\leq
\sum_{l=0}^{2L_0}\sum_{k=3}^{2L_0}\frac{D_3}{M}\bfrac{D_2}{M}^{k+l-1}\\
\leq \sum_{l=0}^{2L_0}\sum_{k=3}^{2L_0}
a_0^{k+l}\leq N^{5\d_0\ln a_0}.
\end{multline}
where
$$\deg(v)\text{ denotes the degree of vertex $v\in V$ in the graph $G_F$.}$$

Markov's inequality implies that there are fewer than $N^{10\d_0\ln a_0}$
such vertices w.h.p.

{\bf Explanation of \eqref{pa}:} We choose
$v_1,v_2,\ldots,v_k$ as the vertices of the cycle
and $w_1,w_2,\ldots,w_l$ as the vertices of
a path joining the cycle at $v_1$. The probability that
the implied edges exist in $K_F$ can be bounded by
\begin{multline*}
\frac{\deg(v_1)\deg(v_2)}{2M-1}\cdot \frac{(\deg(v_2)-1)\deg(v_3)}{2M-3}\cdots
\frac{(\deg(v_k)-1)(\deg(v_1)-1)}{2M-2k+1}\cdot\\
\frac{(\deg(v_1)-2)\deg(w_1)}{2M-2k-1}\cdot \frac{(\deg(w_1)-1)\deg(w_2)}{2M-2k-3}
\cdots \frac{(\deg(w_{k-1})-1)\deg(w_k)}{2M-2l-2k+1}
\end{multline*}
(b) If the condition in (b) fails then there
exist two small cycles that are close together.
More precisely,
there exists a path $P=(v_1,v_2,\ldots,v_k)$ where $k\leq 5L_0$ plus
two additional edges $(v_1,v_i)$ and $(v_k,v_j)$ where
$1<i,j<k$. The probability that such a path exists can be bounded by
\begin{multline}\label{pb}
\sum_{k=4}^{5L_0}\sum_{1<i,j<k}\sum_{\substack{v_1,\ldots,v_k}}\
\frac{\deg(v_1)\deg(v_i)}{M}\cdot \frac{\deg(v_k)\deg(v_j)}{M}\cdot
\prod_{l=1}^k\frac{\deg(v_l)^2}{M}\leq
\sum_{k=4}^{5L_0}\frac{k^2D_3^2D_2^{k-1}}{M^{k+2}}\\
=O(N^{o(1)-1})=o(1).
\end{multline}
Part (b) follows.
\end{proof}
\subsection{Conductance}\label{conductance1}
Given a  connected graph $G=(V,E)$ let
$\p(v)=\frac{\deg(v)}{2|E|}$ denote the steady state probability of being at $v$.
The {\em conductance} $\F(G)$  of a random walk
$\cW_u$ on $G$ is defined by
\beq{conductance}
\F(G)=\min_{S:\p(S)\leq 1/2}\F(S)\text{ where }\F(S)=\frac{|\partial S|}{\deg(S)}
\eeq
and where $\deg(S)=\sum_{v\in S}\deg(v)$ and
$\p(S)=\sum_{v\in S}\p(v)$ and \(\partial S\)
denotes the set of
edges with one endpoint in \(S\) and the other not in \(S\).
(We consider the conductance of random walks
on edge-weighted graphs in Section~\ref{surro}).

The following lemma follows directly from Lemma~10 of \cite{ACF}.
\begin{lemma}\label{condK}
Let $\textbf{d}$ be a nice degree sequence.
Let $F$ be chosen uniformly as in
Section~\ref{config}.
Let $K_F$ be the kernel of the associated configuration multi-graph.
Then with probability $1-o(n^{-1/9})$,
 \[
 \F(K_F) \ge \frac{1}{100}.
 \]
 \end{lemma}

Note that $ \F(K_F) \ge 0.01$ implies that $K_F$ and hence $G_F$ is connected.
Using \eqref{psimple} we see that the
probability that $\Gd$ is not connected is $o(n^{-1/9})=o(1)$.

We will now estimate the conductance of $G_F$
using Lemmas~\ref{lem2} (Part~(c)) and~\ref{condK}.
\begin{lemma}\label{lem3}
 Let $\textbf{d}$ be a nice degree sequence.
Let $F$ be chosen uniformly as in
Section \ref{config}.
Let $G_F$ be the associated configuration multi-graph.
Then with probability $1-o(n^{-1/9})$,
 \[
 \F(G_F)=\Omega\bfrac{\xi}{\ln M}.
 \]
\end{lemma}
\begin{proof}[\textbf{\emph{Proof}}]
Consider a set $S\subseteq [n]$ that induces a connected subgraph of $G_F$.
We can restrict our attention to such sets.
Suppose $S$ only contains part of some path $P_e$.
To be specific, suppose $P_e=(v,u_1,\ldots,u_k,w)$ where $v,w$ are of degree
three or more and $u_1,u_2,\ldots,u_k$ are of degree two. $k=1$ is allowed here.
Assume that $v\in S$. Then we wish to eliminate the case where $u_1,u_2,\ldots,
u_l\in S$ and $u_{l+1}\notin S$ where $l<k$.
If we add an edge of $P_e$
that is not contained in $S$ to create $S'$
then $\deg(S')>\deg(S)$ and $|\partial S'|\leq |\partial S|$. Let $S$ {\em conform}
with the kernel if for all $e\in K_F$
we have either (i) $S$ contains all internal vertices of $P_e$ or (ii)  $S$ contains
no internal vertices of $P_e$. Then w.h.p.
\beq{PhiG}
\F(G_F)\geq \min\set{\min_{\substack{\p(S)\leq 1/2\\ S\text{ conforms with }K_F}}
\frac{|\partial S|}{\deg(S)},
\min_{\substack{1/2-\am/m\leq \p(S)\leq 1/2\\ S\text{ conforms with }K_F}}
\frac{|\partial S|}{\deg(S)+2\am}}.
\eeq
The lemma now follows from $\am=o(m)$ and $\deg(S)\leq \am \deg(S\cap V(K_F))$.
\end{proof}

We note  a result from  Jerrum and Sinclair \cite{Sin}, that
\begin{equation}\label{mix}
|P_{u}^{(t)}(x)-\pi_x| \leq (\p_x/\p_u)^{1/2}(1-\Phi^2/2)^t.
\end{equation}

There is a technical point here.
The result \eqref{mix}
assumes that the walk is lazy.
A  lazy walk  moves
to a neighbour with probability 1/2 at any step.
This assumption halves the
conductance. Asymptotically, the cover time is also doubled.
Otherwise,  the lazy assumption has a negligible effect on the
analysis, see Remark~\ref{remlazy}. We will ignore this
assumption for the rest of the paper; and
continue as though there are no lazy steps.

\section{Estimating first visit probabilities}\label{1stvisit}
In this section $G$ denotes a fixed  connected graph with $\n$ vertices and $\m$ edges.
A random walk $\cW_{u}$  is started from a vertex $u$.
Let $\cW_{u}(t)$ be the vertex
reached at step $t$, let $P$ be the matrix of transition probabilities of the walk and let
$P_{u}^{(t)}(v)=\Pr(\cW_{u}(t)=v)$. We assume
that the random walk $\cW_{u}$ on $G$ is ergodic with
 stationary distribution $\pi$, where $\pi_v=\deg(v)/(2\m)$,
and $\deg(v)$ is the degree of vertex
$v$.

Let
\beq{pwd}
d(t)=\max_{u,x\in V}|P_{u}^{(t)}(x)-\pi_x|,
\eeq
and
let $\Tm$ be a positive integer such that for $t\geq \Tm$
\begin{equation}\label{4}
\max_{u,x\in V}|P_{u}^{(t)}(x)-\pi_x| \leq \n^{-10} .
\end{equation}

Consider the  walk $\cW_v$, starting
at {vertex} $v$. Let $r_t={r_t(v)=}\Pr(\cW_v(t)=v)$ be the probability  that this  walk
returns to $v$ at step $t = 0,1,...$\ .
Let
\begin{equation}
\label{Qs}
 R_{\Tm}(z)=\sum_{j=0}^{\Tm-1} r_jz^j
 \end{equation}
and
let
$$R_v= R_{\Tm}(1).$$

A proof of the following lemma can be found in \cite{CFgiant}.
\begin{lemma}\label{MainLemma}
Let $G=(V,E)$ and let $u,v\in V$ be fixed and let $T = \Tm(G)$. Suppose that
\begin{align}
  T \p_v &=o(1)\,,  \label{lem-RtAv-assumption1}\\
  \min_{|z|= 1+\l}| R_{\Tm}(z)|&\geq \th\,\qquad\text{ for some constant }\th>0.
   \label{lem-RtAv-assumption2}
\end{align}
Then there exists a constant $K$ and values
$\psi_1,\psi_2=O(T\p_v)$ such that if
\begin{equation}
\label{lamby}
\l=\frac{1}{ K\Tm }.
\end{equation}
and
\begin{equation}\label{mm5}
p_v=\frac{\p_v}{R_v(1+\psi_1)}\,.
\end{equation}
then for all $t\geq T$,
\begin{equation}
\label{frat}
\Pr_u(\cA_t(v))=\frac{1+\psi_2}{(1+p_v)^{t}} +O(T\p_v e^{-\l t/2})\,.
\end{equation}
where $\cA_t(v)$ is defined in \eqref{aCv}.
\end{lemma}

\begin{Remark}\label{remlazy}
One effect of making the walk lazy is to (asymptotically)
double $R_v$. Later in the
analysis, this would double our upper bound on the
cover time, as it should. Thus it is
legitimate to ignore this technicality required for \eqref{mix}.
\end{Remark}

Using Lemma~\ref{lem3} and~\eqref{mix} we see that we can take
\beq{Tmix}
\Tm(G_F)=\frac{\ln^4M}{\xi^2}.
\eeq
This is a little larger than one might expect
at this stage. We will explain why later.

Lemma~\ref{MainLemma} is our main tool for proving
upper bounds on the cover time.
\section{Upper bounds}
To begin our analysis we let $G=(V,E)$ be a graph
 with $\n=|V|$ and $|E|=O(\n)$. Assume that $\Tm=\Tm(G)\leq \n$.
Let
$$ \t_u(G,\t)=\min\set{t\geq \t:\cW_u\text{ visits
every vertex of $G$ at least once in the
interval }[\t,t]}.$$
Let $U_t$ be
the number of vertices of $G$ which have not been visited by
$\cW_u$ during steps $[\Tm,t]$.  The following holds:
\begin{align}
&\Tc(G,u)\leq\E_u(\tau_c(G,\Tm))\nonumber\\
&\leq  \Tm+\sum_{t \geq \Tm} \Pr_u(\tau_c(G,\Tm) \geq t)\,,
\nonumber\\
 &=\Tm+\sum_{t \geq \Tm} \sum_{w\in V}
\Pr_w( \t_u(G,0)\geq t-\Tm)\Pr_u(\cW_u(\Tm)=w)\nonumber\\
&\leq \Tm+\sum_{t \geq \Tm} \sum_{w\in V}
\p_w\Pr_w( \t_u(G,0)\geq t-\Tm)+E_1\nonumber\\
&\leq 2\Tm+\sum_{t \geq 2\Tm} \sum_{w\in V}
\p_w\Pr_w( \t_u(G,\Tm)\geq t-\Tm)+E_1\nonumber\\
&= 2\Tm+\sum_{t \geq \Tm} \sum_{w\in V}
\p_w\Pr_w( \t_u(G,\Tm)\geq t)+E_1\label{ETG}
\end{align}
where
\begin{multline}\label{TG1}
E_1=\n^{-10}\sum_{t\geq \Tm}\sum_{w\in V}\Pr_w( \t_u(G,0)\geq t-\Tm)\leq \n^{-3}+
\sum_{ t\geq \n^6}\sum_{w\in V}\Pr_w( \t_u(G,0)\geq \n^4)\leq\\
 \n^{-3}+\sum_{ t\geq \n^6}\sum_{w\in V}\brac{1-(\p_w-\n^{-10})}^{t/\Tm}
\leq  \n^{-3}+\sum_{ t\geq \n^6}\sum_{w\in V}e^{-\Omega(t/ \n^5\log^2\n)}=o(1).
\end{multline}
Here we use $O(\n^4\log\n)$ as a crude upper bound on the mixing time $\Tm$.
It is obtained from the fact that the
conductance of the walk is at least $4/\n^2$  and $\p_w=\Omega(1/\n)$ by assumption.

Now
\beq{TG}
\Pr_v(\tau_c(G,\Tm) > t)=\Pr_v(U_t>0)\leq \min\{1,\,\E_v(U_t)\}\,.
\eeq
It follows from \eqref{ETG},\eqref{TG1},\eqref{TG} that for all $t\gg \Tm$
\begin{equation}
\label{shed}
\Tc(G,u) \leq t+o(t)+ \sum_{s \ge t} \sum_w\p_w\E_w(U_s) = t+o(t)+
\sum_{w\in V}\p_w\sum_{v\in V}\sum_{s \geq t}\Pr_w(\cA_s(v)).
\end{equation}
We will choose a value $t$ and then use Lemma~\ref{MainLemma}
to estimate $\Pr_w(\cA_s(v))$ and show that the
double sum is $o(t)$. It then follows that
$\Tc(G,u)\leq t+o(t)$.

The final expression in \eqref{shed} leads us to define
the random variable
$$\Psi(S,t)=\sum_{v\in V,w\in S}\sum_{s \ge t}\p_v\Pr_v(\cA_s(w))$$
for any $S\subseteq V,\,t\geq 0$. (Here $\Psi$ is a random variable
on the space of graphs $G$).

We can use \eqref{shed} if we have a good estimate for $\Pr_v(\cA_s(w))$.
For this we will use Lemma~\ref{MainLemma}. Let
\beq{d1}
\d_1=\d_0/100
\eeq
\subsection{Case (c1): $M^{1-o(1)}\leq \n_2\leq M^{1+\d_1}$}\label{casec1}
We first check that Lemma~\ref{MainLemma} is applicable.
\subsubsection{Conditions of Lemma~\ref{MainLemma} for $G$}\label{conditions0}
{\bf Checking \eqref{lem-RtAv-assumption1} for $G_F$}:\\
By assumption, the maximum degree in
$G_F$ is at most $N^{o(1)}$.
So for $v\in [n]$ we have from \eqref{Tmix},
$$\Tm \p_v\leb \frac{(M+\n_2)^2\ln^4M}{M^2}\cdot\frac{N^{o(1)}}{M+\n_2}=o(1)$$
where we use $A\leb B$ to denote $A=O(B)$. So,
\eqref{lem-RtAv-assumption1} holds.

\parindent 0in
{\bf Checking \eqref{lem-RtAv-assumption2} for $G_F$:}\\
Suppose that $v$ is one of the vertices that
are placed on an edge $f=(w_1,w_2)$ of $\tKer$. We will say
that $f$ {\em contains } $v$.
We allow $v=w_1$ here and then for convenience we say that
$v$ is contained in one of the edges
incident with $v$ of $K_F$.
We remind the reader that w.h.p.\ all $\tKer$-neighborhoods
up to depth $2L_0$ contain at most
one cycle, see Lemma~\ref{non}(b).
Let $X_f$ be the set of kernel vertices that are
within kernel distance $L_0$ of $f$ in $\tKer$.
Let $\La_f$ be the sub-graph of $G$ obtained as follows:
Let
$H_f$ be the subgraph of the kernel induced by $X_f$. Thus
$f$ is an edge of  $H_f$.
To create $\La_f$ add the vertices of degree two to the edges
of $H_f$ as in the construction of $G_F$.
The vertices of $X_f$ that are at kernel distance $L_0$ from $f$
in $\tKer$ are said to be at the frontier of $\La_f$.
Denote these vertices by $\F_f$.

In this paper we consider walks on several distinct graphs.
We have for example, $\cW_v$, the
random walk on $\tTC$, starting at $v$.
We will now write this as $\CW{\tTC}_v$. The idea of this notation is to
identify explicitly the graph on
which the walk is defined.

Let us make $\F_f$ into absorbing states for a walk
$\CW{\La_f}_v$ in $\La_f$, starting at $v$. Let
$\b(z)=\sum_{t=1}^{\Tm} \b_t z^t$ where $\b_t$ is the
probability of a {\em first} return to  $v$  at time
$t\le \Tm=\Tm(G_F)$ before reaching $\F_f$.
Let $\a(z)=1/(1-\b(z))$, and write
$\a(z)=\sum_{t=0}^{\infty} \a_t z^t$,
so that  $\a_t$ is the probability that the walk
$\CW{\La_f}_v$ is at $v$ at time $t$.
We will prove below that the radius of convergence
of $\a(z)$ is at least
$1+\l$, where $\l$ is as in \eqref{lamby}.

We can write
\begin{eqnarray}
 R_{\Tm}(z)&=&\a(z)+Q(z)\label{BBBB}\\
&=&\frac{1}{1-\b(z)}+Q(z),\label{Bs}
\end{eqnarray}
where $Q(z)=Q_1(z)+Q_2(z)$, and
\begin{eqnarray*}
Q_1(z)&=&\sum_{t=1}^{\Tm}(r_t-\a_t)z^t\\
Q_2(z)&=&-\sum_{t=\Tm+1}^\infty \a_t z^t.
\end{eqnarray*}
We claim that the expression \eqref{Bs} is well defined for
$|z|\le 1+\l$.
We will show  below that
\begin{equation}\label{Q2}
|Q_2(z)|=o(1)
\end{equation}
for $|z|\leq 1+2\l$ and thus the radius of convergence of $Q_2(z)$
(and hence $\a(z)$) is
greater than $1+\l$.
This will imply that $|\b(z)|<1$ for $|z|\leq
1+\l$. For suppose there exists $z_0$ such that $|\b(z_0)|\geq 1$.
Then $\b(|z_0|)\geq |\b(z_0)|\geq 1$ and we can
assume (by scaling) that $\b(|z_0|)=1$. We have
$\b(0)=0<1$ and so we can assume that $\b(|z|)<1$ for $0\leq
|z|<|z_0|$. But as $\r$ approaches 1 from below, \eqref{BBBB} is valid
for $z=\r|z_0|$ and then $| R_{\Tm}(\r|z_0|)|\to \infty$, contradiction.

Recall that $\l=1/ K \Tm$.
Clearly $\b(1)\le 1$ (from its definition) and so  for $|z|\leq 1+\l$
$$\b(|z|)\leq\b(1+\l)\leq\b(1)(1+\l)^{\Tm}\leq e^{1/ K}.$$
Using $|1/(1-\b(z))| \ge 1/(1+\b(|z|))$ we obtain
\begin{equation}\label{Bs<1}
|R_{\Tm}(z)|\geq \frac{1}{1+\b(|z|)}-|Q(z)|
\geq \frac{1}{1+e^{1/ K}}-|Q(z)|.
\end{equation}
We now prove that $|Q(z)|=o(1)$ for $|z|\le 1+\l$ and we will
have verified both conditions of Lemma~\ref{MainLemma}.

Turning our attention first to $Q_1(z)$, we note that
 $r_t-\a_t$ is at most the probability of a return to
$v$ within time $\Tm$, after a visit to $\F_f$ for the walk $\CW{G_F}_v$.
\begin{lemma}\label{plm}
Fix $w\in \F_f$. Then
$$\Pr(\CW{G_F}_w\text{ visits $f$ within time $\Tm$})=O(N^{-\d_0/5}).$$
\end{lemma}
\begin{proof}[\textbf{\emph{Proof}}]
Now consider the walk $\cW_{w}$. We will find an
upper bound for the probability that it reaches $w_1$ or $w_2$, the endpoints
of the $K_f$ edge that $v$ was added to. We consider a simple
random walk $\cX$ on $H$
that starts at $w$ and is reflected when it reaches $\F_f$. We show that
\beq{Xwalk1}
\Pr(\cX\text{ reaches $w_1$ within time $\Tm$})\leq N^{-\d_0/6}.
\eeq
Let $P$ be one of the at most two paths $P,P'$
from $w$ to $w_1$ in $\tKer$. $P=P'$ whenever $w_1$ is locally tree like.
Now to get to $w_1$ the walk $\cX$ will have to traverse the complete length of one
of two paths, $P$ say.
We can ignore the times taken up in excursions outside $P$.
So, we will
think of $\cX$ as a walk along a path in which there are $L_0$ points
at which the probability of moving away from $w_1$
is (at least) 2/3 as opposed to 1/2.
(There could be a couple of places $\g_1,\g_2$
where $P$ meets $P'$ and then we will have the particle
moving further or closer to $w_1$ with different probabilities).
We can also assume that $\ell_e=1$ for all $e\in P$.
This follows from an application of Rayleigh's principle (see, e.g.,~\cite{DS}).
We are reducing the resistance
of $P$ by increasing the conductance of individual edges.
This will increase the (escape)
probability of the
walk reaching $w_1$ before returning to $w$.
(Alternatively we can couple the original walk with
a walk where we have contracted some edges).

So we next consider a biassed random walk
$\cY$ on $[0,L_0]$ where $\cY$ starts at 0 and moves right with probability 1/3.
It follows from Feller \cite[p314]{Fe} that
\begin{equation}\label{abso}
\Pr(\cY\text{ reaches $L_0$ before returning to 0})
\leq\frac{1}{2^{L_0-2}-1}\leq N^{-\d_0/2}.
\end{equation}
(We write $L_0-2$ instead of $L_0$ to account for
the two possible places $\g_1,\g_2$,
where we can just insist on a move towards $w_1$).

Let $N_0=N^{\d_0/4}$. If we restart $\cX$ from $w$ then the probability that
we reach $w_1$ after $N_0$ restarts is at most $N_0N^{-\d_0/2}= N^{-\d_0/4}$.
We observe that $\Tm=O(N^{2\d_1}\ln^4N)\leq N^{\d_0/40}$,
see \eqref{delta0}, \eqref{Tmix} and \eqref{d1}.
To summarise,
\beq{w1}
\Pr(\cW_w\text{ reaches $w_1$ within time }\Tm)\leq \Tm N^{-\d_0/4}\leq N^{-\d_0/5}.
\eeq
By doubling the above estimate in \eqref{w1} to handle $w_2$, we obtain the lemma.
\end{proof}

Thus,
\begin{equation}\label{Q1}
|Q_1(z)| \leq(1+\l)^{\Tm} Q_1(1) \leq 2(1+\l)^{\Tm} N^{-\d_0/5}\Tm=o(1).
\end{equation}

We next turn our attention to $Q_2(z)$.
Let
$\s_t$ be the probability that
the walk on $\La_f$ has not been absorbed by step $t$.
Then $\s_t\geq \a_t$, and so
\[
|Q_2(z)| \leq \sum_{t=\Tm+1}^\infty \s_t |z|^t,
\]

For each $w\in \F_f$ there are one or two paths from $v$ to $w$.
We first consider the number of edges in such a path.
It follows from Part~(c) of Lemma~\ref{lem2} that we can
assume that the number of edges in such a path is
$L\leq L_0\am$.

Assume first that $v$ is locally tree like.
The distance from $v$ of our walk on $\La_f$ dominates the
distance from the origin of a simple random walk on
$\set{0,\pm1,\pm2,\ldots,}$ starting at 0.
We estimate an upper bound for $\s_t$ as follows:
Consider a simple random walk $X_0^{(b)},X_1^{(b)},\ldots$
starting at $|b|< L$ on the finite line $(-L,-L+1,...,0,1,...,L)$,  with
absorbing states  $-L,L$.

$X_m^{(0)}$ is the sum of $m$ independent $\pm 1$ random variables.
So the Central Limit Theorem implies that there exists a constant $c>0$
such that
$$\Pr(X_{cL^2}^{(0)} \geq L \text{ or } X_{cL^2}^{(0)}\leq -L)
\geq 1-e^{-1/2}.$$
Consequently, for any $b$ with $|b|<L$,
\begin{equation}\label{central}
\Pr(|X_{2cL^2}^{(b)}|\geq L)\geq 1-e^{-1}.
\end{equation}
Hence, for $t>0$,
\begin{equation}\label{at}
\s_t \leq \Pr(|X_{\t}^{(0)}|< L,\,\t=0,1,\ldots,t)\leq e^{-\rdown{t/(2cL^2)}}.
\end{equation}
Thus the radius of convergence of $Q_2(z)$ is at least $e^{1/(3cL^2)}$.
As $L\le 4L_0\xi^{-1}\ln M$ we have $L^2\ll \Tm$, see \eqref{Tmix}.
(The need for $L^2\ll \Tm$ explains the larger value
of $\Tm$ than one might expect in \eqref{Tmix}).
So $e^{1/(3cL^2)} \ge
1+2\l$ and for $|z|\leq 1+2\l$,
$$|Q_2(z)|\leq \sum_{t=\Tm+1}^\infty e^{2\l t-\rdown{t/(2cL^2)}}=o(1).$$
This lower bounds the radius of convergence of $\a(z)$ by $1+2\l$,
proves \eqref{Q2} and then \eqref{Q2}, \eqref{Bs<1} and
\eqref{Q1} complete the proof of the case when $v$
is locally tree like.

We now turn to the case where  $\La_f$ contains a unique cycle $C$.
The place where we have used the fact that $\La_f$ is a tree is in
\eqref{at} which relies on \eqref{central}.
Let $x$ be the furthest vertex of $C$ from $v$ in $\La_f$. This is the
only possible place where the  random walk is more likely
to get closer to $v_1$ at the next step. We can see this by
considering the breadth first construction of $\La_f$. Thus we can compare our walk with
 random walk on $[-L,L]$ where there is a unique value $d<L$ such that only
at $\pm d$ is the walk more likely to move towards the origin and
even then this probability is at most 2/3. The
distance of the walk $\CW{\La_f}_{v}$ from $v$
is dominated
by the distance to the origin of a simple random walk, modified
at one of two symmetric places $P_1,P_2$ to move
towards the origin with probability 2/3 instead of 1/2.
A simple coupling shows that making $P_1,P_2=\pm1$
keeps the particle closest to the origin.
We can then contract $0,\pm1$ into one node $0'$ with a loop.
When at $0'$ the loop is chosen with probability 2/3.
The net effect is to multiply the time spent at the origin
by 3, in expectation. We can couple this with a simple
random walk by replacing excursions from the origin and back
by a loop traversal, with probability 2/3. In this way,
 we reduce to the locally tree like case with $\Tm$ inflated
by 4 to account for the loop replacements.

We have now established that in the current case, $G_F$
satisfies the conditions of Lemma~\ref{MainLemma}.

\subsubsection{Analysis of a random walk on $G_F$}\label{RWGF0}
We have a fixed vertex $u\in V$ and a vertex $v$ and we
estimate an upper bound for
$\Pr(\cA_t(v))$ using Lemma~\ref{MainLemma}. For this we need a good upper bound on $R_v$.
Let $f=(w_1,w_2)$ be the edge of $\tKer$ containing $v$.

We write $R_{v}=R_{v}'+R_{v}''$ where $R_{v}'$ is the expected number
of returns to $v$ within time $\Tm$ before the first visit
to $\F_f$ and $R_{v}''$ is the
expected number of visits after the first such visit.

\beq{resist0}
R_v'= \deg(v) R_P
\eeq
where $R_P$ is the {\em effective resistance} (see, e.g., Levin, Peres and Wilmer \cite{LPW})
of a
network $N_{v}$ obtained from $\La_f$ by giving each edge of this graph resistance one
and then
joining the vertices in $\F_f$ via edges of resistance
zero to a common dummy vertex.

For future reference, we note that \eqref{resist0} can be replaced by
\beq{resist00}
R_v'= \l(v) R_P
\eeq
when edges have weight $\l(e)$ and vertices have weight equal to the weight of incidence
edges and edges are chosen with probability proportional to weight.

If $f$ is locally tree like,
let $\wT_1,\wT_2$ be the trees in $\tKer$ rooted at $w_1,w_2$
obtained by deleting the edge $f$ from
$H_f$.
We then prune away edges of the trees $\wT_1,\wT_2$ to make the branching
factors of the two trees exactly two,
except at the root.
We have to be careful here not to delete any edges
incident with the roots. Thus one of the trees
might have a branching factor at the root that is more than two.
Then let $T_1,T_2$ be obtained from $\wT_1,\wT_2$
by placing vertices of degree two on their edges.
If $f$ is not locally tree like then we can
remove an edge of the unique cycle $C$ in $H_f$
not incident with $v$ from $\La_v$ and obtain trees $\wT_1,\wT_2$
in this way. Having done this, we prune edges and add
vertices of degree two to create $T_1,T_2$ as in the
locally tree like case. Removing an edge of $C$ can only increase effective
resistance and $R_v$.

Let $R_1,R_2$ be the resistances of the pruned trees.

We have
$$\frac{1}{R_P}=\frac{1}{\ell_1+R_1}+\frac{1}{\ell_2+R_2}.$$
Here $\ell_i$ is the number of edges in the path from $v$ to $w_i$ in $G_f$.
If $v$ is a vertex of $\tKer$ then we can dispense with $\ell_2,R_2$.

Now when $v\notin V(K_F)$ we have, with $\ell=\ell_1+\ell_2$ and $R=R_1+R_2$,
\beq{arith}
\frac{1}{\ell_1+R_1}+\frac{1}{\ell_2+R_2}\geq \frac{4}{\ell+R}
\eeq
which follows from the
arithmetic-harmonic mean inequality.

When $v\in V(K_F)$ we have
$$\frac{1}{R_P}=\frac{1}{\ell_1+R_1}+\frac{1}{\ell_2+R_2}+
\cdots+\frac{1}{\ell_d+R_d}\geq \frac{d^2}{\ell+R},$$
where $d=d(v)\geq 3$ and $\ell_i$ is the length of the $i$th induced path incident with
$v$ and $R_i$ is the resistance of the tree at the other end of the path.

Let $\cEm$ be the event that $\ell_e\leq\am$ for all $e\in E(K_F)$.
With $\e$ as defined in \eqref{epsdef},
\beq{ind0}
\Pr(R_1\geq\r_1,R_2\geq\r_2,\ell_1+\ell_2=l)\leq
(1+\e) \wPr(R_1\geq\r_1)\wPr(R_2\geq\r_2)\wPr(\ell_1+\ell_2=l).
\eeq
This follows from Part~(a) of Lemma~\ref{lem2}. If
$\om\in \set{R_1\geq\r_1,R_2\geq\r_2,\ell_1+\ell_2=l}$
then $k(\om)\leq 3^{L_0}=M^{o(1)}=o(M)$.
Also, if $\cEm$ holds then
$\z(\om)\leq k\am$ and so $k\z=M^{o(1)}/\xi=o(\n_2+M)$.
Since $\set{R_1\geq\r_1},
\set{R_2\geq\r_2}$, $\set{\ell_1+\ell_2=l}$ depend on
disjoint sets of edges, we can write
the product on the RHS of \eqref{ind0}.

We will implicitly condition on $\cEm$ when using $\Pr$ and
this can only inflate probability estimates
by $1+o(1)$.

We will show in Section \ref{Rv0} that
\beq{R1g0}
\wPr(R_1\geq \r)\leb
\begin{cases}
1&\r\leq L_0\\3^{L_0}(1-\xi)^{\r-2}&\r> L_0
\end{cases}
\eeq
Note that $1-\xi$ can be as small as $N^{-o(1)}$ and so
we cannot replace $(1-\xi)^{\r-2}$
by $(1-\xi)^\r$ without further justification.

We will show in Section \ref{frontier0} that
\beq{R''small0}
R_{v}''=o(R_{v}').
\eeq

Let $Z_{\ell,\r_1,\r_2}$ be the random variable
that is equal to the
number of vertices of $G_F$ with parameters
$\ell=\ell_1+\ell_2,R_1\geq\r_1,R_2\geq\r_2$.
Then we have
\beq{numpaths0}
\E(Z_{\ell,\r_1,\r_2})\leb \sum_{v\in  V(G_F)}\xi(1-\xi)^{\ell-4}
\times 3^{2L_0}(1-\xi)^{\l_1\r_1+\l_2\r_2}.
\eeq
where $\l_i=1_{\r_i\geq L_0}$ for $i=1,2$.

For these vertices, we estimate that, with $\r=\r_1+\r_2$,
\beq{Avest}
\Pr_w(\cA_s(v))\leq \exp\set{-(1+o(1))\frac{\deg(v)}{2m}
\cdot s\cdot\frac{1}{\deg(v)}\cdot
\frac{4}{\ell+\r}}+O(\Tm\p_{\max} e^{-\l t/2})
\eeq
using Lemma~\ref{MainLemma} combined with \eqref{resist0}, \eqref{resist00} and \eqref{arith} to bound
$$\frac{1}{R_v}\geq (1-o(1))\frac{1}{\deg(v)}\cdot\frac{4}{\ell+\r}.$$
Using Lemma~\ref{MainLemma} we see that, where $m=M+\n_2=|E(G_F)|$,
\begin{multline}\label{10}
\E(\Psi(V,t))\leb \frac{3^{2L_0}\xi}{(1-\xi)^4}\sum_{v\in V(G)}
\sum_{s\geq t}\sum_{\ell}
\int_{\r_1,\r_2}d_{\r_1}d_{\r_2} (1-\xi)^{\ell+\r_1\l_1+\r_2\l_2}\times\\
\brac{ \exp\set{-(1+o(1))\frac{\deg(v)}{2m}\cdot s\cdot\frac{1}{\deg(v)}\cdot
\frac{4}{\ell+\r}}+O(\Tm\p_{\max} e^{-\l t/2})}.
\end{multline}
where $\p_{\max}=\max\set{\p_v:v\in V}$.

This is to be compared with the expression in \eqref{shed}.
Here we are summing our estimate for $\Pr(\cA_s(v))$
over vertices $v$.
Notice that the sum over $w\in V$ can be taken
care of by the fact that we weight the
contributions involving $w$ by $\p_w$. Remember that here $w$
represents the vertex reached by $\cW_u$ at time
$\Tm$.

We next remark that with $t=\Omega\bfrac{m\ln^2M}{ -\ln(1-\xi)}$ the term
$$O(\Tm\p_{\max} e^{-\l t/2})=o(e^{-\Omega(M^{1-o(1)}})$$
can be neglected from now on.

We then have
\begin{align}
&\E(\Psi(V,t))\nonumber\\
&\leb \!\!\!\sum_{v\in V(G)}\frac{3^{2L_0}\xi}{(1-\xi)^{4+2L_0}}
\sum_{s\geq t}\sum_{\ell}\int_{\r_1,\r_2}\!\!\!\!\!d_{\r_1}d_{\r_2}
\exp\set{(1+o(1))\brac{(\ell+\r_1\l_2+\r_2\l_2)
\ln(1-\xi)-\frac{2s}{m(\ell+\r)}}}\nonumber\\
&\leb \sum_{v\in V(G)}\frac{3^{2L_0}\xi}{(1-\xi)^4}
\sum_{\ell}\int_{\r_1,\r_2} d_{\r_1}d_{\r_2}
\frac{\exp\set{(1+o(1))\brac{(\ell+\r_1\l_2+\r_2\l_2)\ln(1-\xi)
-\frac{2t}{m(\ell+\r)}}}}
{1-\exp\set{-\frac{2+o(1)}{m(\ell+\r)}}}.\label{tum0}
\end{align}
Our estimate for $\Tc$ is $\Omega\bfrac{m\ln^2M}{-\ln(1-\xi)}$.
So, the contribution from
$\ell_1,\ell_2,\r_1,\r_2$ with $\ell+\r\leq \frac{\g\ln M}{-\ln(1-\xi)}$
is negligible for small enough $\g$. If
$\ell+\r\geq \frac{\g\ln M}{-\ln(1-\xi)}$
then $\ell+\r_1\l_2+\r_2\l_2\sim \ell+\r$, where $A\sim B$
denotes $A=(1+o(1))B$ as $N\to\infty$.
Finally observe that the contributions from $\ell+\r\geq
\frac{\g^{-1}\ln M}{-\ln(1-\xi)}$
will also be negligible.

Ignoring negligible values we obtain a bound by further
replacing the denominator in \eqref{tum0} by $\Omega\bfrac{-\ln(1-\xi)}{m\ln M}$.
Thus,
\begin{align}
&\E(\Psi(V,t))\nonumber\\
&\leb \sum_{v\in V(G)}\frac{m\ln M}{-\ln(1-\xi)}\times
\frac{3^{2L_0}\xi}{(1-\xi)^4}
\sum_{\ell}\int_{\r_1,\r_2}d_{\r_1}d_{\r_2}
\exp\set{(1+o(1))(\ell+\r)\ln(1-\xi)-\frac{2t}{m(\ell+\r)}}
\nonumber\\
&\leb \sum_{v\in V(G)}\frac{m\ln M}{-\ln(1-\xi)}\times \frac{3^{2L_0}\xi}{(1-\xi)^4}
\sum_{\ell}\int_{\r_1,\r_2}d_{\r_1}d_{\r_2}
\exp\set{-\sqrt{\frac{(8+o(1))(-\ln(1-\xi))t}{m}}}\nonumber\\
&\leb M^{2+o(1)}\exp\set{-\sqrt{\frac{(8+o(1))(-\ln(1-\xi))t}{m}}}.
\label{last0}
\end{align}
Putting $t\sim\frac{m\ln^2M}{8(-\ln(1-\xi))}$, where
the implied $o(1)$ term goes to zero
sufficiently slowly, we see that the RHS of \eqref{last0} is $o(t)$.
(Note that $L_0=o(\ln M)$ and $\am,(1-\xi)^{-1},(-\ln(1-\xi))^{-1}=M^{o(1)}$ here).

Summarising, if
\beq{sumTcov0}
t\geq \frac{(1+o(1))m\ln^2M}{8(-\ln(1-\xi))}
\eeq
then
$$\E(\Psi(V,t))=o(t)$$
and then Markov's inequality implies that w.h.p.
$$\Psi(V,t)=o(t).$$

This completes the proof of the upper bound for Case~(c1) of Theorem~\ref{th1},
modulo some claims about $R_{v}$.

\subsubsection{Estimating $R_P$}\label{Rv0}
Assume first of all that we are in the locally tree like case.
We consider the trees $T_1,T_2$.
Their main variability is in the number of vertices
of degree two that are planted on the edges of $\wT_1,\wT_2$.
Fortunately, we only need to compute an upper bound on $\Pr(R(T)\geq \r)$
where $R(T)$ is the resistance of one of these trees.
We focus on $T_1$. Now let the subtrees of
$T_1$ be $T_{1,1},\ldots,T_{1,d}$, where $d\geq 2$.

We have
\beq{4540}
\frac{1}{R(T_1)}=\frac{1}{\ell( T_{1,1})+R(T_{1,1})}+
\cdots+\frac{1}{\ell(T_{1,d})+R(T_{1,d})}
\geq \frac{1}{\ell( T_{1,1})+R(T_{1,1})}+
\frac{1}{\ell(T_{1,2})+R(T_{1,2})}
\eeq
where $\ell_i=\ell( T_{1,i}),\,i=1,\ldots,d$ is the
resistance of the path in $G_f$ from the root of
$T_1$ to the root of $T_{1,i}$.

It follows from this that
\beq{eq30}
\wPr(R(T_1)\geq \r)\leq 2\wPr(\ell_1+R(T_{1,1})\geq 2\r)
\wPr(\ell_2+R(T_{1,2})\geq \r).
\eeq
This is because if $R(T_1)\geq \r$ then (i)
both of the $R(T_{1,i})+\ell_i,\,i=1,2$
must be at least $\r$ and (ii) at least
one of them must be at least $2\r$.

Now, 
\beq{Y10}
\wPr(\ell_1=\ell)=\xi(1-\xi)^{\ell-1}
\eeq
and
\beq{Y20}
\wPr(\ell_1\geq\ell)\leq(1-\xi)^{\ell-1}.
\eeq
Let the level of a tree like $T_1$ be the depth
of the tree in $K_F$ from which
it is derived.
Let $R_k$ be the (random) resistance of a tree
of level $k$, obtained from a binary tree of
depth $k$ by the addition of a random number of vertices
of degree two to each edge.
Putting $R_0=0$ we get from \eqref{eq30} and \eqref{Y20}
that
\beq{eq1eq10}
\wPr(R_1\geq \r)\leq 2(1-\xi)^{3\r-2}.
\eeq
Assume inductively that for $k\geq 1$ and $\r\geq 1$,
\beq{eq40}
\wPr(R_k\geq \r)\leq a_k(1-\xi)^{2\r-k}
\eeq
where
$a_k=(2.5)^k$.

This is true for $k=1$ by \eqref{eq1eq10}.
Using \eqref{eq30} we get that
\begin{align}
\wPr(R_{k+1}\geq \r)&\leq 2\brac{\sum_{s=1}^{2\r-1}\wPr(\ell_1=s)\wPr(R_k\geq 2\r-s)+
\wPr(\ell_1\geq 2\r)}\nonumber\\
&\leq2\brac{\sum_{s=1}^{2\r-1}\xi(1-\xi)^{s-1}
\times a_k(1-\xi)^{2(2\r-s)-k}+(1-\xi)^{2\r}}\label{nn10}\\
&=2\brac{a_k\xi(1-\xi)^{4\r-k-1}\sum_{s=1}^{2\r-1}(1-\xi)^{-s}+
(1-\xi)^{2\r}}\nonumber\\
&\leq 2(a_k+1)(1-\xi)^{2\r-k-1}.\nonumber\\
&\leq a_{k+1}(1-\xi)^{2\r-k-1}.\nonumber
\end{align}

This verifies the inductive step for \eqref{eq40}
and \eqref{R1g0} follows after taking $k=L_0$, with room to spare.

For the non locally tree like case, the deletion of a cycle edge of
$H_f$ to make a tree $\wT_1$, say,
may create one or two vertices of degree two out of kernel vertices.
After adding a random number of
degree two vertices to each edge of $\wT_1$ to
create $T_1$ we will in essence have
created at most two paths whose path length is
(asymptotically) distributed as the sum of two
independent copies
of $Z$, see Lemma~\ref{lem1}. (Such a path arises by concatenating
the two paths $P_{e},P_{e'}$
for a pair of edges $e,e'$ that are incident with a vertex of degree two of $\wT_1$).
We claim that the resistance of such a tree is maximised in distribution
if such
paths are incident with the root and the rest of the paths have a distribution as
in the tree-like-case.
For this we consider moving some resistance $\e$ from one edge closer to the root:
$$\brac{a+\e+\frac{(b-\e)c}{b-\e+c}}-\brac{a+\frac{bc}{b+c}}=
\e\brac{1-\frac{c^2}{(b-\e+c)(b+c)}}\geq 0$$
for $\e\leq b$. Here we have an edge $(x,y)$ of resistance $a$ and two edges of
resistance $b,c$ incident
to $y$ before moving $\e$ of resistance.

The resistance $R$ of $k+1$ levels of such a tree now satisfies
\beq{zxcvb}
\frac{1}{R}=\frac{1}{\r_1'+\r_1''+S_1}+\frac{1}{\r_2'+\r_2''+S_2}
\eeq
where $S_1,S_2$ are copies of $R_k$ and $\r_1',\r_1'',\r_2',\r_2''$ are
copies of $Z$.

Now we will use
\beq{zzxx}
\wPr(\r_1'+\r_1''=\r)\leq 2\wPr(\r_1'\geq \r/2)
\leq2(1-\xi)^{\r/2-1}\text{ and }\wPr(\r_1'+\r_1''\geq 2\r)\leq 2(1-\xi)^{\r-1}.
\eeq
and so arguing as for \eqref{eq30} and \eqref{nn10},
with $\r\geq L$, and using \eqref{eq40},
\begin{align*}
\wPr(R_L\geq \r)&\leb \sum_{s=1}^{2\r-1}(1-\xi)^{s/2-1} (2.5)^L
(1-\xi)^{2(2\r-s)-1}+(1-\xi)^{\r-1}\\
&\leb(2.5)^L(1-\xi)^{\r-2}+(1-\xi)^{\r-1}\\
&\leb  (2.5)^L(1-\xi)^{\r-2}.
\end{align*}
This completes the verification of \eqref{R1g0}.
\subsubsection{Estimating $R_{v}''$}\label{frontier0}

It follows from \eqref{w1} that
$$R_v''\leq N^{-\d_0/5}(R_v'+R_v'')$$
and hence
\beq{summ0}
R_{v}''\leq N^{-\d_0/6}R_{v}'.
\eeq

The proof of the upper bound for Case~(c1) of Theorem~\ref{th1} is now complete.

For the next case we let
$$\om=N^{\z_1}$$
where \eqref{delta0} holds and
\beq{delta1}
\z_0\ll\z_1=o(\d_0)\text{ and now }\d_0\z_1\log N\gg 1.
\eeq
\subsection{Case (c2): $M^{1+\d_1}\leq \n_2\leq e^\om$}\label{casec2}
We recommend that the reader re-visits Section
\ref{n2L}, where we give an outline of our
approach to this case.

It is worth pointing out that
$$\xi=o(1)$$
in this case.

We will be considering several graphs in addition to $G_F$ and $K_F$ and
so it will be important to keep track of their edge and vertex sets. For now
let
$$V_F=V(G_F),E_F=E(G_F)\text{ and }V_K=V(K_F),E_K=E(K_F).$$

We see an immediate problem in the case where
$\n_2/M\to\infty$ too fast. In this case we have
\beq{mm6}
\Tm\p_v=\Omega\brac{\frac{\ln^4M}{\xi^2}\cdot
\frac{1}{\n_2}}=\Omega\bfrac{\n_2\ln^4M}{M^2}.
\eeq
So if $\n_2\geq M^2$ then we cannot apply Lemma~\ref{MainLemma} directly.
Our main problem has been to find a way around this.

We let
\beq{ar1}
\ell^*=\rdown{\frac{1}{\xi\om}}.
\eeq
We begin with some structural properties tailored to this case.
\subsubsection{Structural Properties}\label{struct}

\begin{lemma}\label{lem4}\
W.h.p.\ there is no set $S\subseteq V_K, |S|\leq n_0=N^{1-5000\z_0}$
such that $e(S)\geq (1.001)|S|$.
\end{lemma}
\begin{proof}[\textbf{\emph{Proof}}]
\ignore{
(a)
The expected number of such triangles is at most
$$3\sum_{v_1,v_2,v_3\in V_F}
\prod_{i=1}^3\frac{\deg(v_i)^2}{2M}\cdot\xi\leb a_0^3\xi=o(1).$$}
The expected number of such sets can be bounded by
\begin{align}
\sum_{s=4}^{n_0}\sum_{|S|=s}\binom{\deg(S)}{(1.001)s}\bfrac{\deg(S)}{M}^{(1.001)s}
&\leq \sum_{s=4}^{n_0}\sum_{|S|=s}\brac{\frac{e\deg(S)}{(1.001)s}
\cdot \frac{\deg(S)}{M}}^{(1.001)s}\label{explan}\\
&\leq \sum_{s=4}^{n_0}\binom{N}{s}\bfrac{esN^{2\z_0}}{M}^{(1.001)s}\nonumber\\
&\leq \sum_{s=4}^{n_0}\bfrac{e^{2.001}N^{3\z_0}s^{0.001}}{M^{0.001}}^s\nonumber\\
&=o(1).\nonumber
\end{align}
{\bf Explanation for \eqref{explan}}:
Having chosen a set $X$ of $(1.001)s$ configuration points
for $(1.001)s$ distinct edges, we randomly pair them with other
configuration points.
After pairing $i$ of them, the probability the next point
makes an edge in $S$ using only one point of $X$
is $\frac{\deg(S)-(1.001)s-i}{2M-2i-1}\leq
\frac{\deg(S)}{M}$.
\end{proof}

 An edge $e$ of $K_F$ is {\em light} if $\bm\leq \ell_e\leq \ell^*$.
Let
\begin{align*}
&\wE_\s=\set{e\in E_K:e\text{ is light}}\\
&\wV_\s=\set{v\in V_K:\exists e\in \wE_\s\text{ s.t. }v\in e}\\
\end{align*}
Note that
$$\Pr(e\in \wE_\s)\leq \xi\ell^*\leq \frac1\om.$$
\begin{lemma}\label{degVsigma}
$$\deg(\wV_\s)\leq \frac{2N}{\om^{1/3}},\qquad
\text{with probability at least }1-\om^{-1/3}.$$
\end{lemma}
\begin{proof}[\textbf{\emph{Proof}}]
For any value $D$ we have
$$\E\brac{\card{\set{v\in \wV_\s:\deg(v)\leq D}}}\leq
\frac{D\card{\set{v\in V:\deg(v)\leq D}}}{\om}\leq \frac{ND}{\om}.$$
Putting $D=\om^{1/3}$ and applying Markov's inequality
we see that with probability
at least $1-\om^{-1/3}$.
$$\sum_{v\in \wV_\s:\deg(v)\leq \om^{1/3}}\deg(v)\leq \frac{N}{\om^{1/3}}.$$
In addition we have
$$D^2\sum_{j\geq D}\n_jj\leq D_3\text{ and so}\sum_{v\in \wV_\s:
\deg(v)\geq \om^{1/3}}\deg(v)\leq \frac{D_3}{\om^{2/3}}
\leq \frac{a_0D_1}{\om^{2/3}}\leq \frac{2a_0^{3/2}N}{\om^{2/3}},$$
where we have used \eqref{JJ}.
\end{proof}

Now define a sequence $X_0=\wV_\s,X_1,X_2,\ldots,$ where
$X_{i+1}=X_i\cup\set{x_{i+1}}$ and $x_{i+1}$ is
any vertex
in $V_K\setminus X_i$
that has at least two neighbours in $X_i$. This continues until
we find $k$ for which every vertex in $V_0\setminus X_k$
has at most
one neighbour in $X_k$. Let $\n_0=|X_0|\leq \frac{2N}{\om^{1/3}}$ w.h.p.
Then $X_i$ has $\n_0+i$ vertices and
at least $2i$ edges. Now \eqref{delta1} implies that $\n_0=o(n_0)$
(of Lemma~\ref{lem4})
and so if $i\geq \n_0$ then we contradict the claim in
Lemma~\ref{lem4}. We let
\beq{V1}
V_\s=X_k\text{ and }V_\l=V_K\setminus V_\s
\eeq
and observe that
\beq{dvsig}
|V_\s|\leq \frac{ 4N}{\om^{1/3}}\text{ and so }
\deg(V_\s)\leq D_\s\text{  where }D_\s=\frac{6N^{1+\z_0}}{\om^{1/3}}.
\eeq
Note also that $V_\s$ is well defined in the sense that all sequences
$x_1,x_2,\ldots,$ lead to the same final set.

We will see in Remark~\ref{rem1} why we need $V_\s$ instead of the
simpler $\wV_\s$.

\begin{lemma}\label{frqt}\
W.h.p.\ there is no path of length $L_0$ in $K_F$ with
more than $L_0/10$ members of $V_\s$.
\ignore{
\item W.h.p.\ there is no cycle of length at most $2L_0$ within distance at
most $L_0$ from $V_\s$.
\item W.h.p.\ there is no path of length $L_0$ in $K_F$ that
contains more than $L_0/10$ edges $e$ for which
$\ell_e\leq \ell_1=\om\ell^*/\om_1$, where
$$\om_1=e^{K/\d_0}\text{ for a sufficiently large }K.$$
}
\end{lemma}
\begin{proof}[\textbf{\emph{Proof}}]
First note that if $v_1,v_2,\ldots,v_s\in V_\s$ then there is an ordering
such that $v_1,v_2,\ldots,v_s$ appears as a sub-sequence of
$x_1,x_2,\ldots,x_k$ above. We will assume this ordering and inflate our final
estimate by $s!$ to account for the choice.

We continue by asserting (justification below) that for vertices
$v_1,v_2,\ldots,v_s,s\leq L_0$,
\beq{simplex}
\Pr(v_1,v_2,\ldots,v_s\in V_\s\mid \deg(V_\s)\leq D_\s)
\leq \bfrac{20sN^{6\z_0}}{\om^{2/3}}^s.
\eeq
Thus, given $\cD=\set{\deg(V_\s)\leq D_\s}$,
the expected number of paths in question is bounded by
\begin{multline*}
\sum_{v_1,\ldots,v_{L_0+1}\in V_K}
\prod_{i=1}^{L_0}\frac{\deg(v_i)\deg(v_{i+1})}{2M}
\binom{L_0}{L_0/10}\bfrac{20L_0N^{6\z_0}}{\om^{2/3}}^{L_0/10}\leq\\
\sum_{v_1,\ldots,v_{L_0+1}}\frac{\deg(v_1)\deg(v_{L_0+1})}{M}
\prod_{i=2}^{L_0}\frac{\deg(v_i)^2}{M}\bfrac{200L_0eN^{6\z_0}}{\om^{2/3}}^{L_0/10}\\
\leq \frac{D_1^2D_2^{L_0-1}}{M^{L_0}}\bfrac{200L_0eN^{6\z_0}}{\om^{2/3}}^{L_0/10}
\leb N\bfrac{200L_0ea_0^{10}N^{6\z_0}}{\om^{2/3}}^{L_0/10}=o(1),
\end{multline*}
after using \eqref{delta1}.

{\bf Proof of \eqref{simplex}}: 
Observe first of all that
\begin{align}
& \Pr(v_{i+1}\in \wV_\s\mid v_1,v_2,\ldots,v_i\in V_\s,\cD)\nonumber\\
&=\Pr(v_{i+1}\in \wV_\s\mid v_1,v_2,\ldots,v_i\in \wV_\s,\cD)
\Pr(v_1,v_2,\ldots,v_i\in \wV_\s,\cD\mid v_1,v_2,\ldots,v_i\in V_\s,\cD)\nonumber\\
&\leq \Pr(v_{i+1}\in \wV_\s\mid v_1,v_2,\ldots,v_i\in \wV_\s,\cD)\nonumber\\
&\leq \frac{iN^{\z_0}}{M}+\Pr(v_{i+1}\in \wV_\s\mid v_1,v_2,\ldots,v_i\in \wV_\s,\cD,
(v_{i+1},v_j)\notin \wE_\s,\forall j)\label{ms2}\\
&\leq \frac{iN^{\z_0}}{M}+\frac{N^{\z_0}}{\omega}\nonumber\\
&\leq \frac{2N^{\z_0}}{\omega}.\label{ms3}
\end{align}
{\bf Explanation of \eqref{ms2}:} The first term $iN^{\z_0}/M$ is a bound
on the probability that $v_{i+1}$ is a
neighbour of some $v_j,j<i$.
The second term is a bound on the probability that an edge
incident with $v_{i+1}$ is light. We deal with the conditioning by first exposing
$K_F$ and then exposing the placement of the vertices of degree two.

We will now prove that
\beq{ms4}
\Pr(v_{i+1}\in V_\s\setminus \wV_\s\mid  v_1,v_2,\ldots,v_i\in V_\s)\leq
\frac{18N^{6\z_0}}{\om^{2/3}}.
\eeq
Recall that we assume the order $v_1,v_2,\ldots,v_i$ is such
that $v_j$ can be placed in
$V_\s$ once $v_1,v_2,\ldots,v_{j-1}$ have been so placed. Then, using the notation of
Section \ref{config}, we let
$\wW=W\setminus W_{v_{i+1}}$. If $|W_{v_{i+1}}|$ is odd, we first choose a
random point $x\in \wW$ and pair up the
remainder of points to create $\wF$. Suppose now that $W_{v_{i+1}}=
\set{x_1,x_2,\ldots,x_k}$.
We define a sequence of configuration
multi-graphs
$\G_0=\wK_{\wF},\G_1,\ldots,\G_k=K_F$. We obtain
$\G_{j+1}$ from $\G_j$ as follows:
If $k-j$ is odd then we pair up $x_j$
with the unpaired point in $\G_j$. If $k-j$ is even
we choose a random pair $\set{y,z}$ in $\G_j$
and pair $x_{j+1}$
with $y$ or $z$ equally likely, leaving the other point unpaired.

We first claim that $\G_0,\G_1,\ldots,\G_k$ are all
random pairings of their respective  point sets.
We do this by induction. It is trivially true for $\G_0$.
When $k-j$ is odd, the construction
is equivalent to choosing a random point to pair with
$x_{j+1}$ and then choosing a random
configuration ($\G_j$) on the remaining points. If
$k-j$ is even, then we again pair $x_{j+1}$
with a random point $y$, say. Then $z$ will be a
uniform random point and the remaining
configuration will be a random pairing of what is left.

Assume that $\deg(V_\s(\G_0))\leq D_\s$.
Now $v_{i+1}$ will be placed into $V_\s(\G_k)$ only
if there are two values of $j$ for
which $x_{j+1}$ is paired with a
point associated with a vertex in $V_\s(\G_j)$.
Up to this point we will have $V_\s(\G_j)\subseteq V_\s(\G_0)$.
It follows that $x_{j+1}$ is so
paired with probability
at most
\beq{ms6}
\binom{k}{2}\bfrac{D_\s}{M}^2.
\eeq
Equation \eqref{ms4} (and the lemma) follows from \eqref{ms6}, after
inflating the final estimate by $s!$.
\end{proof}

Consider the following property of $S\subseteq V_\l$ (defined
in \eqref{V1}): Let $s=|S|$. 
\beq{shou}
\text{(i) $S$ induces a tree in $K_F$; (ii) $\deg(S)\leq s\ln N$;
(iii) $e(S:V_\s)\geq \eta_s=\max\set{3,\rdup{s/500}}$.}
\eeq
\begin{lemma}\label{shou1}
W.h.p., if $S$ satisfies \eqref{shou} then $|S|\leq s_1$ where
$$s_1=\frac{10000\ln N}{\ln \om}.$$
\end{lemma}
\begin{proof}[\textbf{\emph{Proof}}]
Let $Z_s$ be the number of sets satisfying \eqref{shou} under these circumstances.
Assume that $s>s_1$. Then, from \eqref{dvsig},
\beq{whatsiT}
\E(X_s)\leq (1+o(1))\sum_{\substack{|S|=s\geq s_1\\\deg(S)\leq s\ln N}}
\binom{\deg(S)}{s/500}
\bfrac{D_\s}{M}^{s/500}\binom{\deg(S)}{s-1} \bfrac{\deg(S)}{M}^{s-1}.
\eeq
{\bf Explanation:}
We choose configuration points that will be paired
with $V_\s$ in $\binom{\deg(S)}{s/500}$ ways.
The probability that all these points are paired in $V_\s$ is at most
$$\bfrac{\deg(V_\s)}{2M-\deg(S)}^{s/500}\leq \bfrac{D_\s}{2M-\deg(S)}^{s/500},$$
see Lemma
\ref{degVsigma}.
We choose $s-1$ configuration points for the edges
inside $S$. The probability they are paired with
other points associated with $S$
can be bounded by $\bfrac{\deg(S)}{2M-o(M)}^{s-1}$.
The factor $1+o(1)$ arises from the conditioning
imposed by assuming \eqref{dvsig}. Also, after conditioning on $V_\s$
we only allow a vertex in $V_\l$ to choose a single neighbour in $V_\s$.
Thus $\binom{d(S)}{s/500}$ is an over-estimate of the number of choices.

Continuing,
\begin{align*}
\E(X_s)&\leq \sum_{\substack{|S|=s\geq s_1\\\deg(S)\leq s\ln N}}(500e\ln N)^{s/500}
\bfrac{6N^{\z_0}}{\om^{1/3}}^{s/500}(e\ln N)^s\bfrac{s\ln N}{N}^{s-1}\\
&\leq \bfrac{Ne}{s}^s(500e\ln N)^{s/500}
\bfrac{6N^{\z_0}}{\om^{1/3}}^{s/500}(e\ln N)^s\bfrac{s\ln N}{N}^{s-1}\\
&\leb N\brac{\frac{CN^{\z_0/500}\ln^{2.002+o(1)}N}{\om^{1/1500}}}^{s-1},
\qquad C=e^{2+o(1)}(3000e)^{(1+o(1))/500},
\end{align*}
 see \eqref{delta1}.

So,
$$\E\brac{\sum_{s\geq s_1}X_s}\leb N\sum_{s\geq s_1}
\brac{\frac{CN^{\z_0/500}\ln^2N}{\om^{1/1500}}}^{s-1}=o(1).$$

This implies that w.h.p.\ we have $X_s=0$ for $s\geq s_1$.
\end{proof}

We now wish to show that small sets of $K_F$-edges do not contain too many
vertices of degree two.
\begin{lemma}\label{smallweight}
W.h.p.\ no subset $S\subseteq E_K$ satisfies
$|S|\leq \e M$ and $\ell(S)=\sum_{e\in S}\ell_e\geq L=\e^{1/2}M/\xi$.
provided $\e$ is a sufficiently small positive constant.
In particular, this holds for any $\e\leq \e_{1}$ where
$\e_{1}$ is the solution to $\e^{3/2}e^{1/ (6\e^{1/2})}=20$.
\end{lemma}
\begin{proof}[\textbf{\emph{Proof}}]
Let $S$ be ``bad'' if it violates the conditions of the lemma.
We can assume w.l.o.g. that $|S|=\e M$ here.
Now using
\eqref{830} to go from the first line to the second,
\begin{align}
\Pr(\exists\text{ a bad }S)&\leq \binom{M}{\e M}\sum_{\ell=L}^m
\binom{\ell-1}{\e M-1}
\frac{\binom{\n_2+M-1-\ell}{M-\e M-1}}{\binom{\n_2+M-1}{M-1}}\nonumber\\
&\leq\sum_{\ell=L}^m\bfrac{Me}{\e M}^{\e M}\bfrac{\ell e}{\e M}^{\e M}
\xi^{\e M}(1-\xi)^{\ell-\e M}\brac{1+\frac{(1+o(1))\e M}{\n_2}}^{\ell -\e M}
\nonumber\\
&=\sum_{\ell=L}^m\brac{\frac{e^2\ell\xi}{\e^2M(1-\xi)}
\brac{1+\frac{(1+o(1))\e M}{\n_2}}^{-1}}^{\e M}\brac{(1-\xi)
\brac{1+\frac{(1+o(1))\e M}{\n_2}}}^\ell\nonumber\\
&\leq \sum_{\ell=L}^m\bfrac{10\ell\xi}{\e^2M}^{\e M}(1-(1-2\e)\xi)^\ell.
\label{831}
\end{align}
Putting $\ell=AM/\xi$ into the summand $u_\ell$ of \eqref{831} we
obtain for sufficiently small $\e$ that
\beq{832}
u_\ell\leq\bfrac{10Ae^{-A/(2\e)}}{\e^2}^{\e M}\leq e^{-\e^{1/2}M/3}.
\eeq
Now $A\geq \e^{1/2}$ and a quick check shows that \eqref{832}
is valid if $\e^{3/2}e^{ 1/(6\e^{1/2})}\geq 10$.

So,
$$\Pr(\exists\text{ a bad }S)\leq me^{-\e^{1/2}M/3}=o(1),$$
given our upper bound of $e^{M^{o(1)}}$ for $m$.
\end{proof}

The next lemma shows that our assumption on degrees implies that a small
set of vertices has small total degree.
\begin{lemma}\label{smalldegree}
If $S\subseteq V_K$ and $|V|\leq \e N$ then $\deg(S)\leq 2a_0\e^{1/3}N$
for $\e<1$.
\end{lemma}
\begin{proof}[\textbf{\emph{Proof}}]
Let $S_0=[N_\e,N]$ where $N_\e=N-\e N+1$.
It is enough to prove the lemma for $S=S_0$.
Let $D_\e=\sum_{j\in S_0}d_j$ and $L=d_{N_\e}$. Then
\beq{900}
D_\e\leq \sum_{k\geq L}k\n_k\leq \sum_{k\geq L}
\frac{k^2}{L}\n_k\leq \frac{a_0N}{L}.
\eeq
If $L> 1/\e^{1/3}$ then we are done and so assume that $L\leq 1/\e^{1/3}$.

Let $S_1=\set{j:d_j\geq L/\e^{1/3}}$. Then, following the argument in \eqref{900}
for $S_1$ we get
$$D_\e\leq \frac{\e NL}{\e^{1/3}}+\sum_{j\in D_1}d_j\leq  \frac{\e NL}{\e^{1/3}}
+\frac{a_0\e^{1/3}N}{L}$$
and the result follows.
\end{proof}

\subsubsection{Surrogates for $G_F$}\label{surro}
We have seen that we can use \eqref{shed} if we have
a good estimate for $\Pr_v(\cA_s(w))$.
We have seen in \eqref{mm6} that we cannot
necessarily apply the lemma directly in this case.
So what we will do
is find a graph $G$ that satisfies the conditions of Lemma~\ref{MainLemma}
and whose cover time is related
in some easily computable way to the cover time in $G_F$.
(This statement is only approximately true, but it can be used as motivation for
some of what follows).

In the following, we define graphs that will be surrogates
for $\tTC$ with respect to computing the
cover time.

Let $e$ be an edge of $\tKer$.
We will break the corresponding path $P_e$ of
length $\ell_e=p_e\ell^*+q_e,\,p_e\geq 0,0\leq q_e<\ell^*$
in the graph $\tTC$ into consecutive
sub-paths  $Q_f,\,f\in F_e$.
For a typical path, where $p_e\geq 1$
there will be $p_e-1$
paths of length $\ell^*$ and one path of length $\ell^*+q_e$.
There will however
be some cases where $e$ is light and so we have to be a little more careful.
When $e$ is light we do nothing to $P_e$. In
this case, $P_e$ is considered as a sub-path
of itself in the following and is replaced by a
single edge in the graph $G_0$ defined below.
Otherwise we construct $p_e-1$ paths of length $\ell^*$
and one path of length $\ell^*+q_e$.
Let $\cQ_e$ denote the set of
sub-paths created from $P_e$.

We define the graph $G_0=(V_0,E_0)$ as follows: For each $e\in E_K$,
we replace each sub-path $Q\in\cQ_e$
of length $\ell_Q$ by an edge $f=f_Q$ of {\em weight or conductivity}
$\k(f)=\ell^*/\ell_Q$.
The {\em resistance} $\r(f)$ of edge $f$
is given by $1/\k(f)$. Note that the total resistance of a
heavy edge $e$ is $\ell_e/\ell^*$.

We will use the notation $f\in e$ to indicate
that edge $f$ of $G_0$ is obtained from a sub-path
of edge $e\in E_K$.


We now check that the total weight of the edges in $G_0$
is what we would expect.
We remark first that since $M=o(\n_2)$ and $M=\Theta(N)$ we have
$$m\sim |V(G_F)|\sim \n_2.$$
\begin{lemma}\label{E(G_0)}
W.h.p.,
$$\k(E_0)\sim |E_0|\sim\frac{|E(\tTC)|}{\ell^*}=\frac{\n_2+M}{\ell^*}\sim \om M.$$
\end{lemma}
\begin{proof}[\textbf{\emph{Proof}}]
Each edge $e\in E_K$ gives rise to a path of length $\ell_e$ in
$\tTC$.
We let
$$K_0=\set{e\in E_K:\ell_e<\ell^*},\;K_1=
\set{e\in E_K:\ell^*\leq \ell_e<\ell^* \om^{1/3}}\text{ and }
K_2=E_0\setminus (K_0\cup K_1).$$
Then,
\begin{align}
|E_0|&=\frac{1}{\ell^*}\sum_{e\in K_1\cup K_2}(\ell_e-q_e)
+|K_0|\label{3DD}\\
&=\frac{m}{\ell^*}-\frac{|K_0|}{\ell^*}-\frac{1}{\ell^*}\sum_{e\in K_1\cup K_2}q_e
+|K_0|.\label{3D}
\end{align}
Now for $e\in E_K$ and $0\leq q<\ell^*$, and using Part~(b) of Lemma~\ref{lem2}
with $k=1,\z=q$,
$$\Pr(q_e=q)\sim\sum_{r\geq 0}\xi(1-\xi)^{r\ell^*+q-1}=\xi(1-\xi)^{q-1}
\cdot\frac{1}{1-(1-\xi)^{\ell^*}}
\sim\om\xi(1-\xi)^{q-1}.$$
So,
$$\E(q_e)\sim\sum_{k=1}^{\ell^*-1}k\om\xi(1-\xi)^{k-1}
\leq \frac{\om\xi}{(1-\xi)^2}\leq \ell^*,$$
and
\beq{4D}
\E\brac{\sum_{e\in E_K}q_e}\leb \ell^*M.
\eeq
So w.h.p.
\beq{5D}
\frac{1}{\ell^*}\sum_{e\in K_1\cup K_2}q_e=o(M\om^{1/2}).
\eeq
Now
$$\E(|K_0|)\sim M\sum_{q=1}^{\ell^*-1}\xi(1-\xi)^{q-1}=
O(\ell^*\xi M)=O(M/\om).$$
So,
\beq{6D}
|K_0|=o(M)\ w.h.p.
\eeq
Going back to \eqref{3D} with \eqref{5D} and \eqref{6D} and
$$\frac{m}{\ell^*}\sim \om M$$
we see that our expression for $|E_0|$ is correct, w.h.p.

Now w.h.p.
\ignore{
\begin{align*}
\k(E_0)&=\sum_{e\in K_1\cup K_2}\brac{\frac{\ell^*(p_e-\b_e)}{\ell^*+\a_e}+
\frac{\ell^*\b_e}{\ell^*+\a_e+1}}+\sum_{e\in K_0}\frac{\ell^*}{\ell_e}\\
&=\sum_{e\in K_1\cup K_2}\brac{p_e-
\frac{\a_ep_e}{\ell^*+\a_e}+\frac{\ell^*\b_e}{(\ell^*+\a_e)(\ell^*+\a_e+1)}}
+\sum_{e\in K_0}\frac{\ell^*}{\ell_e}\\
&=\sum_{e\in K_1\cup K_2}\brac{\frac{\ell_e}{\ell^*}-\frac{q_e}{\ell^*}
-\frac{\a_ep_e}{\ell^*+\a_e}+\frac{\ell^*\b_e}{(\ell^*+\a_e)(\ell^*+\a_e+1)}}
+\sum_{e\in K_0}\frac{\ell^*}{\ell_e}\\
&=\frac{m}{\ell^*}-\sum_{e\in K_1\cup K_2}\brac{\frac{q_e}{\ell^*}+
\frac{\a_ep_e}{\ell^*+\a_e}-\frac{\ell^*\b_e}{(\ell^*+\a_e)(\ell^*+\a_e+1)}}+
\sum_{e\in K_0}\brac{\frac{\ell^*}{\ell_e}+\frac{\ell_e}{\ell^*}}.
\end{align*}
}
\begin{align*}
\k(E_0)&=\sum_{e\in K_1\cup K_2}\brac{p_e-1+
\frac{\ell^*}{\ell^*+q_e}}+\sum_{e\in K_0}\frac{\ell^*}{\ell_e}\\
&=\sum_{e\in K_1\cup K_2}\brac{p_e-
\frac{q_e}{\ell^*+q_e}}+\sum_{e\in K_0}\frac{\ell^*}{\ell_e}\\
&=\frac{m}{\ell^*}-\sum_{e\in K_1\cup K_2}\frac{q_e}{\ell^*+q_e}
+
\sum_{e\in K_0}\brac{\frac{\ell^*}{\ell_e}+\frac{\ell_e}{\ell^*}}.
\end{align*}
To finish the proof we show that the terms other
than $m/\ell^*$ contribute $o(\om M)$ in expectation
and then we can apply Markov's inequality.
We can use \eqref{5D} to deal with the first sum.
We are left with
\begin{multline*}
\E\brac{\sum_{e\in K_0}\frac{\ell^*}{\ell_e}}=
\ell^*\sum_{e\in E_K}\sum_{k=1}^{\ell^*-1}
\frac{\Pr(\ell_e=k)}{k}\leq \\
(1+o(1))\ell^*M\brac{\sum_{k=1}^{\n_2^{1/3}}\frac{\xi(1-\xi)^{k-1}}{k}+
\sum_{k=\n_2^{1/3}}^{\n_2}\frac{\binom{M+\n_2-k-1}{M-2}}{k\binom{M+\n_2-1}{M-1}}}\\
\leb\frac{M\ln\n_2}{\om}+\frac{M^2}{\xi\n_2}
\exp\set{-\frac{(M-2)\n_2^{1/3}}{M+\n_2-1}}
=o(\om M),
\end{multline*}
where to get the final expression we have used the
calculations in Part~(c) of Lemma~\ref{lem2},  i.e., \eqref{lmax}.
Of course we can use \eqref{6D} to deal with
$\sum_{e\in K_0}\ell_e/\ell^*\leq |K_0|$.
\end{proof}

Since, from \eqref{3DD} and the above analysis,
$$\sum_{e\in K_1\cup K_2}(p_e-1)\leq |V_0|\leq
|E_0|\leq \sum_{e\in K_1\cup K_2}p_e+o(\om M) $$
we have that w.h.p.
$$|V_0|\sim |E_0|\sim \om M.$$

We will analyse the expected time for a random walk $\CW{G_0}$
on $G_0$ to cross each edge of $G_0$ at least once.
We will be able to couple this with $\CW{G_F\to V_0}$,
the projection of $\cW^{G_F}$ onto $V_0$. We will see below that
if either walk is at $v\in V_0$ and $w$ is a neighbour of $v$ in
$G_0$ then $w$ has the same probability of being the next $V_0$-vertex
visited in both walks.

It is easy to see that
after $\CW{G_0}$ has crossed each edge of
$G_0$, in the coupling, $\CW{G_F}$
will have visited each vertex of $G_F$.

We must modify $G_0$ slightly, because we have to cover the
{\em edges} of $G_0$. Let $f^* =(v_1,v_2)$ be an
edge of $G_0$.


The graph $G_0^*=G_0^*(f^*)$ will be obtained from $G_0$ by splitting
$f^*$.
We give edges $(v_1,\x)$ and $(\x,v_2)$ a weight of $\a=\min\set{\a_f,1}$
where $\a_f$ is the weight of edge $f$.

$\GW$ is the random walk on $G_0^*$, where we choose edges according to weight;
$\tW$ is the projection of $\GW$ onto $V_0$.
This walk is $\GW$ with visits to $\x$ omitted from the
sequence of states.
This means that time passes more slowly in $\GW$ than it does in $\tW$.
We use $G_0^*$ in order to deal with the edge cover time of $G_0$,
which is what we need, see \eqref{mm9} below.

Our goal is to compute a good upper
estimate for $\Pr(\cA_s(f^*))$ where
$\cA_s(f^*)$ is the event that we have not crossed edge $f^*$ in the time
interval $[\Tm,s]$. We do this by going to $G_0^*$ and
estimating $\Pr(\cA_s(\x))$ for
the random walk on $G$. Note that $\Pr(\cA_s(f^*))=
\Pr(\cA_s(\x))$ if $f$ is a heavy edge and
$\Pr(\cA_s(f^*))\leq
\Pr(\cA_s(\x))$ if $f$ is a light edge. Indeed, in both cases there is a natural
coupling of $\GW$ and $\sW$, up until $v_1$ or $v_2$ are reached.
This is because walks in $G_0$ and walks in $G_0^*$ that do not contain
$v_1$ or $v_2$ as a middle vertex have the same probability in both.
Having reached $v_1$ or $v_2$ there is no lesser chance of crossing
$f^*$ in $G_0$ than there is of visiting $\x$ in $G_0^*$. In the case of a heavy
edge, we can extend this coupling up until $\x$ is visited. This follows
from our choice of weight for the edges $(v_i,\x),i=1,2$.


There is a problem with respect to using $G_0$ as a surrogate in
that its mixing time can be too large.
If the edges of a graph are weighted then the
conductance of a set of vertices $S$ is given by
$$\F(S)=\frac{\sum_{x\in S,y\in \bar{S}}\k(x,y)}{\k(S)}=
\frac{\k(\partial S)}{\k(S)}.$$
Consider an edge $e=(u,v)\in E_K$ for
which $\ell_e=1$ and such that (i) $u,v$ both have degree three in $K_F$ and
(ii) all edges of $E_K$ other than $e$
incident with $u,v$ are heavy. Let $S=\set{u,v}$.
Then in $G_0$, $\F(S)=O(1/\ell^*)$, making $\F(G_0)$ too small.
The situation cannot be dismissed as only happening with probability $o(1)$.

We remark that if the following conjecture is true, then we will be able to fix the
problem of small edges by adding more vertices of degree two. We will be able to
do this so that $\ell^*$ divides $\ell_e$ for all $e\in E_K$. This would simplify
the proof somewhat.

\begin{conjecture}\label{conj-extra}
Adding extra vertices of degree two to the edges of $K_F$ to make
$\ell_e\geq \ell^*$ for all $e$, does not
decrease the cover time.
\end{conjecture}

In the absence of a proof of this conjecture, we must find a work around.
We observe for later that if every edge $e$ has a
weight $\k(e)\in [\k_L,\k_U]$ then
we have
\beq{phiS}
\F(S)\geq \frac{\k_L\partial S}{\k_U\deg(S)}
\eeq
where $\partial S$ is defined following \eqref{conductance}.

We now define the graph $G$. It will have vertex
set $V_\l^*=V_{\l}\cup\set{v_1,\x,v_2}$, see \eqref{V1}.
A $G_0^*$-edge $f$ contained in $V_{\l}$ will give
rise to an edge of weight $\k_f$ in $G$.

Next let $N_1'$ be the set of vertices in $V_{\l}$
that have $K_F$-neighbors in $V_\s$ and let $N_1=N_1'\cup\set{v_1,\x,v_2}$.
The
edges from $N_1$ to $V_\s$ will also give rise to $G$ edges.
For each $x\in V_\s\cup N_1$ and $y\in N_1$
we define $\th(x,y)$ as follows:
Consider the random walk $\CW{G_0^*}_x$. This starts at $x$ and it
chooses to cross an incident edge of the
current vertex
with probability proportional to its $G_0^*$-edge weight.
Suppose that this walk follows the sequence $x_0=x,x_1\in V_\s,x_2,\ldots,$
and that $k,k\geq 1$ is the smallest positive index such that $x_k\notin V_\s$.
Then, $\th(x,y)=\Pr(x_k=y)$. Then for
$x\in N_1$ and $z\in V_\s$ for which $f=(x,z)$ is an edge of
$G_0^*$ and $y\in N_1$ ($y=x$ is allowed) we add a {\em special} edge,
{\em oriented} from $x$ to $y$
of weight $\k_f\th(z,y)$. We remind the reader that $\k_f=\ell^*/\ell_f$.

We have introduced some orientation to the edges. We need to check that the
Markov chain we have
created is reversible. Then we can use conductance to estimate the mixing time.
In verifying this claim we will see that the steady state of the
walk is proportional to
$\k(x)$ for $x\in V_{\l}$. We do this by checking detailed
balance. For $x,y\in V_{\l}^*$ we let $P(x,y)$
be the probability
of moving in one step from $x$ to $y$. We let $P(x,y)=P_0(x,y)+P_1(x,y)$
where $P_0(x,y)$ is the
probability of following a special edge from $x$ to $y$.
We have $\k(x)P_1(x,y)=\k(y)P_1(y,x)$
because these
quantities are derived from the random walk on $G_0^*$. As for $P_0(x,y)$, we have
\begin{align*}
\k(x)P_0(x,y)&
=\sum_{z_0\in V_\s}\sum_{z_1,z_2\ldots z_l}\k(x)P_1(x,z_0)
\prod_{i=0}^{l-1}P_1(z_i,z_{i+1})\times P_1(z_l,y)\\
&=\sum_{z_0\in V_\s}\sum_{z_1,z_2\ldots z_l}\k(z_0)P_1(z_0,x)
\prod_{i=0}^{l-1}P_1(z_i,z_{i+1})\times P_1(z_l,y)\\
&\ \vdots\\
&=\k(y)P_0(y,x).
\end{align*}
As a further step in the construction of $G$, we remove some loops. In particular,
if $x\in N_1$ and $p=P(x,x)>0$ then
$$P(x,x)\gets 0\text{ and }P(x,y)\gets P(x,y)/(1-p)\text{ for }y\in N_1,y\neq x.$$

Because the chain is reversible we can define an associated electrical network
$\cN$,
which is an undirected graph with an edge $(x,y)$ of weight (conductance)
$C_{x,y}=\k(x)P(x,y)=\k(y)P(y,x)$.

We claim that we can couple $\cX_1=\ttW$ and $\cX_2=\CW{G}$ where
$\ttW$ is the projection of $\CW{G_0^*}$ onto $V_{\l}$. This walk is $\CW{G_0^*}$
with visits to $V_\s$ omitted from the
sequence of states. Indeed, we have designed $G$ so that for each $v,w\in V_{\l}$
$$\Pr(\cX_1(t+1)=w\mid \cX_1(t)=v)=\Pr(\cX_2(t+1)=w\mid \cX_2(t)=v).$$
\ignore{
In terms of time we will charge two steps for following an oriented edge.
To get ``real $\CW{G_0}$
time'' we would have to add time spent by $\CW{G_0}$ crossing edges inside $V_\s$.
One way to do
this is to replace $\CW{G}$ by the walk $\CW{G^*}$
where we carry out $\CW{G_0}$ but stop
the clock after we enter $V_0$ and re-start it just before we leave $V_0$.
With this definition, $\CW{G^*}$ is a random walk on $V_\s$ which is identical
in terms of time taken to cover vertices of $V_{\l}$ as $\CW{G}$.
}
\begin{Remark}\label{rem1}
The reader can now see why we defined $V_\s$ in the way we did.
If we had stopped with $\wV_\s$ then $G_0$ might contain
isolated vertices.
\end{Remark}

{\bf Coupling $\CW{G_0},\CW{G}$ and $\CW{\tTC}$:}

We consider the vertices $V_0$ of $G_0$ to be a subset of the vertices of $\tTC$.
We couple $\CW{\tTC}$ with a random
walk $\sW$ on $G_0$.
In the walk $\sW$ edges are selected with probability proportional
to their weight/conductivity. We will now check that there is a natural
coupling.

Suppose that $\cWa$ is at a vertex $v\in V_0$. Suppose that $v$
has neighbours $w_1,w_2,\ldots,w_d$ in $G_0$
and that $f_i=(v,w_i)$ for $i=1,2,\ldots,d$.
In $\tTC$ there will be corresponding paths
$P_i$ from $v$ to $w_i$.
Let $i^*\in [d]$ be the
index of the path whose other endpoint is next reached by $\CW{\tTC}$.
Then if $\ell(P)$ is the length of a path $P$, we prove below that
\beq{coup}
\Pr(i^*=i)=\frac{\ell(P_i)^{-1}}{\ell(P_1)^{-1}+\cdots+\ell(P_d)^{-1}}=
\frac{\k_i}{\k_1+\cdots+\k_d}
\eeq
where $\k_i=\k(f_i)$.

This can be proved by induction. Let $\ell_i=\ell(P_i),\,i=1,2,\ldots,d$.
Our induction is on $L=\ell_1+\cdots+\ell_d$.
The base case where $\ell_i=1$ for $i=1,2,\ldots,d$ is trivial. Now
suppose that $\ell_1\geq 2$. Then if $\Pi=\Pr(i^*=1)$,
\beq{coup1}
\Pi=\frac{(\ell_1-1)^{-1}}{(\ell_1-1)^{-1}+\ell_2^{-1}+\cdots+
\ell_d^{-1}}\brac{\frac{\ell_1-1}{\ell_1}+\frac{\Pi}{\ell_1}}.
\eeq
{\bf Explanation:} The factor $\frac{(\ell_1-1)^{-1}}{(\ell_1-1)^{-1}+
\ell_2^{-1}+\cdots+\ell_d^{-1}}$
is, by induction, the probability that the walk reaches the penultimate
vertex of $P_1$ and then $\frac{\ell_1-1}{\ell_1}$ is the probability
that the walk reaches the end of $P_1$ before going back to $v$.
The term $\frac{\Pi}{\ell_1}$ is then the probability that $i^*=1$
in the case that the walk returns to $v$.

Equation \eqref{coup} follows from \eqref{coup1} after a little algebra.

Note that \eqref{coup} is the probability that
$\sW$ chooses to move to $w_i$ from $v$.
Thus we see that
$\cWa$ and $\sW$ can be coupled so that they go through the exact
same sequence of vertices in $V_0$, although
$\sW$ moves faster.

The
 expected relative speed of these walks can be handled
with the following lemma.
\begin{lemma}\label{mm7}
Suppose that $T$ is a tree consisting
of a root $v$ and $k$ paths $P_1,P_2,\ldots,P_k$
with common vertex $v$ and no other common vertices. Path $P_i$ has length $\ell_i$
for $i=1,2,\ldots,k$.
A walk $\cW$ starts at $v$.
\begin{enumerate}[(a)]
\item The expected time $\La$
for $\cW$ to reach a leaf is given by
$$\La=\frac{\ell_1+\cdots+\ell_k}{\sum_{i=1}^k\ell_i^{-1}}.$$
\item If $\ell_i\leq \ell$ for $i=1,2,\ldots,k$ then $\La\leq \ell^2$.
\end{enumerate}
\end{lemma}
\begin{proof}[\textbf{\emph{Proof}}]

(a) Observe that
\beq{RW3}
\E(\text{time to reach a leaf})+\E(\text{time back to }v)
=\frac{2(\ell_1+\cdots+\ell_k)}{\sum_{i=1}^k\ell_i^{-1}}.
\eeq
The RHS is twice the number of edges in $T$ times the effective
resistance between $v$ and the set of leaves.
(see e.g. \cite{LPW}, Proposition 10.6)

It follows from \eqref{coup} and the fact that a simple random walk takes $\ell^2$
steps in expectation to move $\ell$ steps in distance that
$$\E(\text{time back to }v)=\sum_{i=1}^k
\frac{\ell_i^{-1}}{\sum_{i=1}^k\ell_i^{-1}}\times \ell_i^2.$$
Part (a) of the lemma follows.

(b) We simply observe that increasing $\ell_i$ increases
the numerator and decreases the
denominator.

This completes the proof.
\end{proof}

We next observe that in this coupling, if $\sW$ has covered all of the
{\em edges} of $G_0$ then $\cWa$ has covered
all of the edges of $\tTC$, and so the edge cover time
of $G_0$, suitably scaled, is an
upper bound on the edge and hence vertex cover time
of $\tTC$.

It follows from Lemma \ref{mm7}(b) and the fact that all sub-paths
have length at most $(1+o(1))\ell^*$ that
that if $D_u$ is the expected time for the walk $\cW_u$ on $\tTC$
to cover all the edges of $\tTC$ and $D_v^*$ is the
expected time for the walk $\sW_v$ on
$G_0$ to cover all the edges of $G_0$, then
\beq{mm9}
\Tc=\max_uC_u\leq \max_uD_u\leq (1+o(1))(\ell^*)^2(\max_vD_v^*+1).
\eeq
(The +1 accounts for the case when $u$ is in the middle of a sub-path).

In the same way, we can couple $\CW{G_0}$ and $\CW{G}$, up until the first visit
to $\x$, in the following sense.
We can consider the latter walk to be the former, where we
ignore visits to $V_\s$. By construction,
if $v\in V_\l, w\in V_{\l}^*$
then for both walks we have that $w$ has
the same probability of being the next
vertex in $V_{\l}^*=V_\l\cup\set{\x}$
that is visited by the walk.
We will show in Section \ref{complete} that the
time spent
in $V_\s$ is negligible.

\subsection{Conditions of Lemma \ref{MainLemma} for $G$}\label{subsec:cond-mainlemma}

{\bf Checking \eqref{lem-RtAv-assumption1} for $G$}:\\
We first claim that we have
\beq{replace}
\Tm(G)=O(\om^{2}\ln^5M).
\eeq
Let $\tG=(V_\l,E_\l)$ be the subgraph of $K_F$ induced by $V_{\l}$.
We begin by estimating the conductance of $\tG$, as in \eqref{conductance}.
Let $\Pi_{\b,s},0\leq \b\leq 1\leq s\leq s_0=\om^{-1/3}N^{1+2\z_0}$
be the probability
that there is a {\em connected} set $S\subseteq V_{\l}$ with
$|S|=s$ and $e_K(S)=\b\deg(S)/2\geq |S|$ and $e_K(S:V_\s)\geq (1-\b)\deg(S)/2$.
(Here $e_K(S)$ is the number of $G_\l$ (or $K_F$)
edges contained in $S$ and $e_K(S:V_\s)$
is the number of edges joining $S$ and $V_\s$ in $K_F$).
\begin{lemma}\label{mn2}
The following holds simultaneously and
w.h.p.\ for every set $S\subseteq V_{\l}$ that induces a connected
subgraph of $\tG$: In the following, $e_\l(S:\bar{S})$ is the number of
$G_\l$ edges joining $S$ to $\bar{S}=V_\l\setminus S$. Note that

\begin{enumerate}[(a)]
\item If (i) $|S|\leq s_0$ and (ii) $e(S)=\b \deg(S)/2\geq |S|$ then
$$e_\l(S:\bar{S})\geq \frac{(1-\b)\deg(S)}{2}.$$
\item If $e(S)=|S|-1$ then
$$e_\l(S:\bar{S})\geq \frac{2\deg(S)}{3s_1},$$
where $s_1=\frac{10000\ln N}{\ln \om}$.
\end{enumerate}
\end{lemma}
\begin{proof}[\textbf{\emph{Proof}}] (a)
We estimate $\Pi_{\b,s}$ from above by
\beq{whatsit}
\Pi_{\b,s}\leq \sum_{|S|=s}
\binom{\deg(S)}{(1-\b)\deg(S)/2}
\bfrac{N^{1-C\z_0}}{M}^{(1-\b)\deg(S)/2}\binom{\deg(S)}{\b\deg(S)/2}
\bfrac{\deg(S)}{M}^{\b\deg(S)/2}.
\eeq
where $C$ can be any positive constant.

{\bf Explanation:}
We choose configuration points that will be paired with $V_\s$ in
$\binom{\deg(S)}{(1-\b)\deg(S)/2}$ ways.
The probability that all these points are paired in $V_\s$ is at most
$$\bfrac{\deg(V_\s)}{2M-\deg(S)}^{(1-\b)\deg(S)/2}\leq
\bfrac{N^{1-C\z_0}}{2M-\deg(S)}^{(1-\b)\deg(S)/2},$$
see \eqref{dvsig}.
We choose $\b\deg(S)/2$ configuration points for the edges inside $S$.
The probability they are paired with
other points associated with $S$
can be bounded by $\bfrac{\deg(S)}{2M-o(M)}^{\b\deg(S)/2}$.

Using \eqref{whatsit} we see that
\begin{align}
\Pi_{\b,s}&\leb \sum_\d\sum_{\substack{|S|=s\\ \deg(S)=\d s}}
\bfrac{2e}{1-\b}^{(1-\b)\d s/2}\bfrac{N^{1-C\z_0}}{M}^{(1-\b)\d s/2}
\bfrac{2e}{\b}^{\b \d s/2}\bfrac{\b \d s}{M}^{\b \d s/2}
\nonumber\\
&\leq \sum_\d\sum_{\substack{|S|=s\\ \deg(S)=\d s}}\brac{2(N^{-C\z_0})^{1-\b}
\bfrac{2e\d s}{N}^{\b}}^{\d s/2}.\label{d(S)}
\end{align}
We first consider the case where $3\leq \d\leq A=N^{\z_0}$.
Let $\th_{\d,s}$ be the proportion of sets of size $s$ that have $\deg(S)=\d s$.
In which case, \eqref{d(S)} becomes
\begin{align}
\Pi_{\b,s}&\leb \sum_\d\th_{\d,s}
\binom{N}{s}\brac{2eN^{-C(1-\b)\z_0}\bfrac{2eAs}{N}^{\b}}^{\d s/2}\nonumber\\
&\leq \sum_\d\th_{\d,s}\brac{2e^2N^{-C(1-\b)\z_0 /2}
\bfrac{s}{N}^{\b/2-1/\d}A^{\b/2}}^{\d s}.
\label{res}
\end{align}
At this point we observe that by assumption, we have $\b\deg(S)/2\geq |S|$ and so
\beq{bd}
\frac{\b\d}{2}\geq 1.
\eeq
Now because $\d\geq 3$ and $\sum_\s\th_{\d,s}=1$,
we have
\begin{align}
&\Pi_{\b,s}\leb \sum_\d\th_{\d,s}\brac{2e^2A^{\b/2}\bfrac{s}{N}^{1/24}}^{\d s}\leq
\bfrac{s}{N}^{s/16}\qquad \text{if }\b\geq 3/4.
\label{noo1}\\
&\Pi_{\b,s}\leb \sum_\d\th_{\d,s}\brac{2e^2A^{\b/2}N^{-C\z_0/4}}^{\d s}\leq
N^{-3C\z_0s/8}
\qquad \text{if }\b\leq 3/4\text{ and }C\geq 2.\nonumber
\end{align}

Now the number of choices for $\b$ can be bounded by $\deg(S)$ and we bound
this by $N^{\z_0}s$. This gives, for this case,
$$
\sum_{\b,s}\Pi_{\b,s}\leq \sum_{s=1}^{s_0}N^{\z_0}s\bfrac{s}{N}^{s/16}
+\sum_{s=1}^{s_0}N^{\z_0}sN^{-3C\z_0s/8}=o(1),
$$
if $C\geq 3$.

We now consider those $S$ for which $\deg(S)\geq A|S|$. Going back to \eqref{d(S)}
we see that for these we have
\begin{align*}
 \Pi_{\b,s}&\leb \sum_\d \th_{\d,s}\binom{N}{s}\brac{2eN^{-C(1-\b)\z_0/2}
\bfrac{2eN^{\z_0}s}{N}^{\b}}^{As/2}\\
&\leq \sum_\d \th_{\d,s}\brac{4e^{1+2/A}N^{-C(1-\b)\z_0}\bfrac{s}{N}^{\b-2/A}
}^{As/2}
\end{align*}
This yields
\begin{align}
&\Pi_{\b,s}\leq \bfrac{s}{N}^{As/5}\qquad \text{if }\b\geq 1/2.\label{noo2}\\
&\Pi_{\b,s}\leq \brac{4e^{1+o(1)}N^{-C\z_0/2}}^{As/2}
\qquad \text{if }\b\leq 1/2.\nonumber
\end{align}
and we can easily see from this that $\sum_{\b,s}\Pi_{\b,s}=o(1)$
in this case too,
for $C\geq 3$. Thus
w.h.p.
$$e(S:V_{\l})=\deg(S)-2e(S)-e(S:V_\s)\\
=\deg(S)-\b\deg(S)-e(S:V_\s)\geq (1-\b)\deg(S)/2.$$
(b)
Now consider sets with $e(S)=|S|-1$ and use Lemma \ref{shou1}.
If $|S|> s_1$ then either $\deg(S)>s\ln N$
or $e(S:V_\s)\leq \rdup{s/500}$. The former implies that
$$\frac{e(S:V_{\l})}{\deg(S)}\geq \frac{\deg(S)-2(|S|-1)-|S|}{\deg(S)}=1-o(1)$$
and the latter implies that
$$\frac{e(S:V_{\l})}{\deg(S)}\geq
\frac{\deg(S)-2(|S|-1)-\rdup{|S|/500}}{\deg(S)}>\frac{249}{250}.$$
If $|S|\leq s_1$ then
and since $\deg(S)\geq 3|S|$,
\begin{equation*}
\frac{e(S:V_{\l})}{\deg(S)}\geq \frac{\deg(S)-2(|S|-1)-|S|}{\deg(S)}\geq
\frac{2}{3|S|}\geq \frac{2}{3s_1}.\qedhere
\end{equation*}
\end{proof}

We verify next that if $S\subseteq V_0$ and $|S|$ is too close to
$N$ then $\k(S)$ will exceed $\k(G)/2$. Suppose then that $|S|\geq (1-\eta)N$
where $2a_0\eta^{1/3}=\e_{1}$ of Lemma \ref{smallweight}. It follows from
Lemma \ref{smalldegree} that $\deg_{K_F}(V_K\setminus S)\leq 2a_0\eta^{1/3}N=
\e_{1}N$.
It then follows from Lemma \ref{smallweight} that
$$\sum_{\substack{e\in E_K\\e\cap S=\emptyset}}\ell_e\leq
\frac{\e_1^{1/2} M}{\xi}\text{ and hence }\sum_{\substack{e\in E_K\\
e\cap S\neq\emptyset}}\ell_e\geq 2m-
\frac{2\e_1^{1/2} M}{\xi}\geq (2-3\e_1^{1/2})m.$$
It follows from this and Lemma \ref{E(G_0)} that
\beq{bigset}
\k(S)\geq \brac{1-\frac{3\e_1^{1/2}}{2}}\k(G_0).
\eeq
It is shown in \cite{ACF} that if $S\subseteq V_K$, then in $K_F$ we have
\beq{mn1}
e(S:V_K\setminus S)\geq \deg(S)/50\text{ for all sets $S$ with $\deg(S)\leq M$}.
\eeq
Now suppose that $S\subseteq V_0$ and $\k(S)\leq \k(G_0)/2$. It follows from
\eqref{bigset} that $|S|\leq (1-\eta)N$. This implies that
$\deg_{K_F}(S)\leq 2M-3\eta N$.

If $\deg_{K_F}(S)\leq M$ then \eqref{mn1} implies that $e(S:\bar{S})
\geq \deg(S)/50$.

If $\deg_{K_F}(S)> M$ then $3\eta N\leq \deg_{K_F}(\bar{S})\leq M$ and hence
$e(S:\bar{S})\geq 3\eta N/50\geq (3\eta/50a_0)\deg(S)$.

It follows that if $\k(S)\leq \k(G_0)/2$ then
\beq{sweat}
e_{\tG}(S:V_\l)\geq \begin{cases}\frac{2d(S)}{3s_1}&|S|\leq s_0\\
\frac{3\eta}{50a_0}\deg(S)-\deg(V_\s)\geq \frac{3\eta}{50a_0}\deg(S)-
\frac{6N^{1+\z_0}}{\om^{1/3}}\geq\frac{2\eta}{50a_0}\deg(S)&s_0< |S|
\leq (1-\eta)N\end{cases}
\eeq
Now every heavy edge of $G_0$ has weight at least 1/2. Applying the
argument for \eqref{PhiG} we see that \eqref{sweat} implies that
$$\F(G_0)=
\min_{\substack{S\subseteq V_0\\\k(S)\leq \frac12\k(V_0)}}\F_{G_0}(S)
=\Omega\bfrac{1}{\am}\times \min_{\substack{S\subseteq V_K\\
|S|\leq (1-\eta)N}}
\frac{e_{\tG}(S:V_\l)}{\deg(S)}=\Omega\bfrac{1}{ \om\ln^2 M}.$$
\ignore{
Let $\wG$ be obtained from $\tG$ by adding
weighted edges between vertices in $N_1$ as we
did just after \eqref{phiS}. In the resulting
graph, the weights of
vertices will be within a bounded factor of their degrees.
This bound lies in
$[1/3+o(1),1]$. This is
because degrees in $V_\l$ are at least three and $v\in V_\l$
has at most one neighbour in $V_\s$. Furthermore,
the weight of an edge of $G_0$ derived from a heavy edge
of $K_F$ is at least 1/2.
So,
\beq{310}
\F_{\wG}(S)\geq \frac{\F_{\tG}(S)}{3+o(1)}\geq \frac{e(S,V_{\l})}{4\deg(S)}.
\eeq
It follows from this that
$$
\F_{\wG}(S:V_\l)\geq \frac1{4}\begin{cases}
           \frac2{3s_1}&|S|\leq s_0\\
           \frac{1}{100}&|S|>s_0
          \end{cases}
$$
This shows that $\F(\wG)=\Omega(1/\ln N)$. We now go
back to the argument for \eqref{PhiG} and argue
that $\F(G)=\Omega(\F(\wG)/(\om\ln M))$. Here we use
the fact that w.h.p no edge of $K_F$ contains
more than $\am=O(\om\ln M\times \ell^*)$ vertices of degree two.
}
Taking account of
the special edges introduced to bypass most of the light edges can only increase the
conductance of a set. This is because it won't
affect the denominator in the definition of conductance, but it
might increase the numerator.

All that is left is to consider
the effect of splitting
the edge $f^*$ into a path of length two in order to define $G_0^*=G_0^*(f^*)$.
The conductance of a connected set $S$ not containing $v_1$ or $v_2$
is not affected by this change. If $S$ contains $v_1,v_2$ then
after the split, the numerator remains the same. On the other hand,
the denominator can at most double. If $S$ contains one of $v_1,v_2$ then
the numerator still remains the same and again the denominator can at most double.

Thus
$\F(G)=\Omega(\F(G_0))$.
Equation \eqref{replace} now follows from $\Tm(G)=O(\F^{-2}\ln M)$.

We then have
\beq{zxc1}
\Tm(G)\p_{G}(\x)=O\bfrac{\om^2\ln^5M}{\om M}=o(1).
\eeq

\parindent 0in
{\bf Checking \eqref{lem-RtAv-assumption2} for $G$:}\\
Let $f^*=(v_1,v_2)$ as before.
Suppose that $v_1$ is one of the vertices that
are placed on a $\tKer$ edge $f=(w_1,w_2)$. We allow $v_1=w_1$ here.
We now remind the reader that w.h.p.\ all $\tKer$-neighborhoods
up to depth $2L_0$ contain at most
one cycle, see Lemma \ref{non}(b).
Let $X$ be the set of kernel vertices that are
within kernel distance $L_0$ of $f$ in $\tKer$.
Let $\La_f$ be the sub-graph of $G$ obtained as follows:
Let
$H$ be the subgraph of the kernel induced by $X$. This definition includes $f$ as
an edge of $H$.
If $H$ contains no members of $V_\s'=V_\s\setminus\set{v_1,v_2}$ then we do nothing.
Otherwise, let $T$ be a component of the subgraph of
$H$ induced by $V_\s'$ and let
$L=\set{v_0,v_0',v_1,\ldots,v_s}\subseteq N_1$ be the neighbours of $T$ in
$V_\l^*$
where $v_0,v_0'$ are the vertices in $L$ that are closest to $\set{w_1,w_2}$.
Here $v_0=v_0'$ is allowed and this is indeed occurs in the majority of cases w.h.p.
Note also that by the construction of $V_\s$,
each $v_i,i\geq 1$ has one neighbour in $T$.
We replace $T$ by special edges $(v_0,v_i),(v_0',v_i),(v_i,v_0),(v_i',v_0)
,i=1,2,\ldots,s$.
If $T$ contains a vertex $w$
that is at distance $L_0$ from
$\set{w_1,w_2}$ then we remove $T$ completely.

Next add vertices of degree two to the non-special edges of
$H$ as in the construction of the 2-core.
We obtain $\La_f$ by contracting paths as in the construction of $G_0$.
Vertices of $X$ that are at maximum kernel distance from $f$
in $\tKer$ are said to be at the frontier of $\La_f$. Denote these vertices by $\F_f$.

We now follow the argument in Section \ref{conditions0} between
``Let us make $\F_f$ into...'' and Lemma \ref{plm},
the proof of which requires some minor tinkering:
\begin{lemma}\label{plm1}
Fix $w\in \F_f$. Then
$$\Pr(\GW_w\text{ visits $f$ within time $\Tm$})=O(N^{-\d_0/2})=o(1).$$
\end{lemma}
\begin{proof}[\textbf{\emph{Proof}}]
Let $P$ be one of the at most two paths $P,P'$
from $w$ to $w_1$ in $\tKer$; then $P=P'$ whenever $w_1$ is locally tree like.
Let $e_1,e_2,\ldots,e_{L_0}$ be the edges of $P$.
Assume first that neither of these paths contain a member of $V_\s$.
We will correct for this later. In this case we can follow the argument
of Lemma \ref{plm} until the end.

Suppose now that the paths contain members of $V_\s$.
It is still true that there are only one or
two paths from boundary vertex $w$ to $w_1$ or $w_2$.
The only change needed for the analysis is
to note that after contracting special edges these $K_F$
paths can shrink in length to $9L_0/10$.
Here we use Lemma \ref{frqt}. This
changes $2^{L_0-2}$ in \eqref{abso} to $2^{9L_0/10-2}$
and allows the proof to go through.
\end{proof}

The remainder of the verification follows as in Section \ref{conditions0}.
\subsubsection{Analysis of a random walk on $G$}
This is similar to the analysis of Section \ref{RWGF0} and may seem a bit
repetitive.  We will first argue that
\beq{mm20}
\text{the edge cover-time of $G$ is w.h.p.
at most }
\frac{\om^{2}M\ln^2M}{8+o(1)}.
\eeq
After this we have to deal with the time spent crossing edges with at least one endpoint
in $V_\s$. This will be done in Section \ref{complete}.

We have a fixed vertex $u\in V_{\l}$ and an edge $f^*$ and we will
estimate an upper bound for
$\Pr(\cA_t(\x))$ using Lemma \ref{MainLemma}.
For this we need a good upper bound on $R_{\x}$.
Let $f=(w_1,w_2)$ be the edge of $\tKer$ containing $f^*$.
Recall the definition of $\La_f$ in Section~\ref{subsec:cond-mainlemma} where we were checking
\eqref{lem-RtAv-assumption2}.
If $f$ is locally tree like
let $T_1,T_2$ be the trees in $G_0$ rooted at
$w_1,w_2$ obtained by deleting the edges of
$\La_f$ that are derived from the edge $f$ of $\tKer$.
If $f$ is not locally tree like then we can
remove an edge of the unique cycle $C$ in $\La_f$
not incident with $\x$ from $\La_f$ and obtain trees $T_1,T_2$
in this way. Removing such an edge can only increase resistance and $R_f$.

We write $R_{\x}=R_{\x}'+R_{\x}''$ where $R_{\x}'$ is the expected number
of returns to $\x$ within time $\Tm$ before the first visit
to $\F_f$ and $R_{\x}''$ is the
expected number of visits after the first such visit.

\beq{resist}
R_{\x}'=2\a R_P
\eeq
where $R_P$ is the effective resistance as defined in Section
\ref{RWGF0}, but associated to the weighted network $\cN$.
Here $\a$ is the weight of the edge $f$ that we split.

We first assume that $\La_f$ contains no vertices in $V_\s$ and then in the final paragraph of Section~\ref{Rv}
we show what adjustments are needed for this case.

We will show in Section \ref{frontier} that
\beq{R''small}
R_{\x}''=o(R_{\x}').
\eeq

We first prune away edges of the trees $T_1,T_2$ tree-like
neighbourhoods to make the branching factor of the associated trees at most two.
Of course, in tree like neighborhoods we can say exactly two.
This only increases the effective resistance and $R_{\x}$.
Let $R_1,R_2$ be the resistances of the pruned trees and let
$R=R_1+R_2$.

We have
\begin{equation}\label{Rw}
\frac{1}{R_P}=\frac{1}{\a^{-1}+\ell_1/\ell^*+R_1}+\frac{1}{\a^{-1}+\ell_2/\ell^*+R_2}.
\end{equation}
Here $\ell_i/\ell^*$ is the total resistance
of the $G$ edges in the path from $v_i$ to $w_i$
derived from $f$.
If $v_1$ is a vertex of $\tKer$ then we can dispense with $\ell_2,R_2$.

Note that, with $\ell=\ell_1+\ell_2$,
\beq{ache}
\frac{1}{\a^{-1}+\ell_1/\ell^*+R_1}+\frac{1}{\a^{-1}+\ell_2/\ell^*+R_2}
\geq \frac{4}{4+\ell/\ell^*+R}
\eeq
(which follows from $\a\geq 1/2$ and the
arithmetic-harmonic mean inequality).

Let $\cEm$ be as defined before \eqref{ind0}
and note that given $\cEm$ we have
$\e=O\bfrac{3^{2L_0}\ln M}{\xi(M+\n_2)}=o(1)$, where $\e$ is defined
in Part~(a) of Lemma~\ref{lem2}.
We re-write \eqref{ind0} as
\beq{ind}
\Pr(R_1\geq\r_1,R_2\geq\r_2,L=(\ell_1+\ell_2)/\ell^*= \ell/\ell^*)\leq (1+\e)
 \wPr(R_1\geq\r_1)\wPr(R_2\geq\r_2)\wPr(\ell_1+\ell_2=l).
\eeq
Note next, that with $\ell=\ell_1+\ell_2$, and given $\a$ and that $\xi=o(1)$,
$$\wPr(L=(\ell_1+\ell_2)/\ell^*=\ell/\ell^*\mid\cEm)
\leq\xi(1-\xi)^{\ell-1}\leb\xi e^{-L/\om}.$$
We will show in Section \ref{Rv} that for $\r=M^{o(1)}$ we have
\beq{R1g}
\wPr(R_1\geq \r\mid \cEm)\leb 3^{L_0}e^{-\r/\om}
\eeq
This is a simpler expression than \eqref{R1g0} because here we have $\xi=o(1)$.

Let $Z_{L,\r_1,\r_2}$ be the random variable that is equal to the
number of vertices of $G_0$ with parameters $L,\r_1,\r_2$.
Then we have
\beq{numpaths}
\E(Z_{L,\r_1,\r_2})\leb \om M\times L\ell^*\times \xi e^{-L/\om} \times
3^{L_0}e^{-R/\om}
= 3^{L_0}\om MLe^{-(L+\r)/\om},
\eeq
where $\r=\r_1+\r_2$. (The factor $\ell_e=L\ell^*$ comes form the number of
choices of edge to split in path $P_e$).

Using Lemma \ref{MainLemma}
and \eqref{ache} we see that
\begin{multline}\label{1}
\E(\Psi(E(G_0),t))\leb 3^{L_0}\om\xi M\sum_{s\geq t}\ell^*\int_LdL
\int_{\r_1,\r_2}d_{\r_1}d_{\r_2} L e^{-(L+\r)/\om}\times\\
\brac{ \exp\set{-(1+o(1))\frac{s}{2\om M}\cdot
 \frac{4}{4+L+\r}}+ O(\Tm^2\p_{\max} e^{-\l t/2})}.
\end{multline}
where $\p_{\max}=\max\set{\p_v:v\in V}$.

{\bf Some explanation:} The first line is direct from \eqref{numpaths}.
Then $\frac{2\a}{2\om M}$ is asymptotic to the steady
state for $\x$ and there is a $\frac{1}{2\a}$ factor from \eqref{resist}.
So $\frac{\p_{\x}}{R_{\x}}$ is asymptotic to $\frac{2\a}{2\om M}\cdot\frac{1}{2\a} \cdot\frac{4}{4+L+\r}=\frac{1}{2\om M}\cdot\frac{4}{4+L+\r}$.

This is to be compared with the expression in \eqref{shed}.
Here we are summing our estimate for $\Pr(\cA_s(f))$
over edges $f$ of weight $\a$. Recall that $\cA_s(f)$ is the
event that we have not crossed edge $f$ in the time
interval $[\Tm,s]$.

Notice that the sum over $v\in V$ can be taken care of by the fact that we weight the
contributions involving $v$ by $\p_v$. Remember that here $v$
represents the vertex reached by $\sW$ at time
$\Tm$.

Ignoring a negligible term we have
\begin{align}
&\E(\Psi(E(G_0)),t))\nonumber\\
&\leb 3^{L_0}\om\xi M\sum_{s\geq t}\ell^*\int_LdL\int_{\r_1,\r_2}d_{\r_1}d_{\r_2} L
\exp\set{-(1+o(1))\brac{\frac{L+\r}{\om}+\frac{2s}{\om M(4+L+\r)}}}\nonumber\\
&\leb 3^{L_0}\om\xi M\ell^*\int_LdL\int_{\r_1,\r_2}d_{\r_1}d_{\r_2}L
\frac{\exp\set{-(1+o(1))\brac{\frac{L+\r}{\om}+\frac{2t}{\om M(L+\r)}}}}
{1-\exp\set{-\frac{2+o(1)}{\om M(L+\r)}}}.\label{tum}
\end{align}
Note now that in the current case, $\xi=o(1)$ and so
our estimate for $\Tc$ is $\sim C\om^{2}M\ln^2M$
where $C\geq 1/8$. So, the contribution from
$\ell,\r$ such that $L+\r\leq \om\ln M/100$ is negligible.
As are the contributions from $L+\r\geq 5\om\ln M$.

Ignoring negligible values we obtain a bound by further
replacing the denominator in \eqref{tum} by $\Omega(1/ \om^2M\ln M)$.
Thus,
\begin{align}
\E(\Psi(E(G_0),t))&\leb 3^{L_0}\om^3\xi \ell^*M^2\ln M
\int_{L\leq 5\om\ln M}\int_{\r\leq 5\om^2\ln M}
L\exp\set{-\frac{L+\r}{\om}-\frac{2t}{\om M(L+\r)}}\nonumber
\\
&\leb 3^{L_0}\om^4 M^2\ln M\times (\om\ln M)^3
\exp\set{-\sqrt{\frac{8t}{\om^2M}}}.\label{last}
\end{align}
Putting $t\sim\frac{1}{8}\om^2M\ln^2M$ we claim that the RHS of \eqref{last} is $o(t)$. Indeed, to see this note that $3^{L_0}\om^4 M^2\ln M\times (\om\ln M)^3=M^{2+\eta}$ for some $\eta=o(1)$, where $M^\eta\to\infty$. Therefore, if we take $t=\frac{1+3\eta}{8}\om^2M\ln^2M$ then the RHS of \eqref{last} is $\leb M^{2+\eta}\times M^{-(1+3\eta)^{1/2}}=o(M)$.

We now consider the contribution of $O(\Tm^2\pi_{\max}
e^{-\l \Tc/2})$ to $\E(\Psi(E(G_0),t))$.
We bound this by
$$\leb (\om^2\ln^5M)^2\times \frac{1}{\om M}\times
\exp\set{-\Omega\bfrac{\om^2M\ln^2M}{\om^2\ln^5M}}=o(1).$$
Summarising, if
\beq{sumTcov}
t\geq \frac{1+o(1)}{8}\om^2M\ln^2M
\eeq
then
$$\E(\Psi(E(G_0),t))=o(t)$$
and then the Markov inequality implies that w.h.p.
$$\Psi(E(G_0),t)=o(t).$$

\subsubsection{Estimating $R_P$}\label{Rv}
We first assume that $\La_f$ contains no vertices from $V_\s$.

We follow the argument in Section \ref{Rv0} down to
\eqref{Y10}, \eqref{Y20} which we replace by
\beq{Y1}
\wPr(\ell_1/\ell^*=\r)=\xi(1-\xi)^{\r\ell^*-1}
\eeq
and
\beq{Y2}
\wPr(\ell_1/\ell^*\geq\r)=(1-\xi)^{\r\ell^*}.
\eeq
Let the level of a tree like $T_1$ be the depth of
the tree in $K_F$ from which it is derived.
Let $R_k$ be the (random) resistance of a tree of level $k$.
Putting $R_0=0$ we get from \eqref{eq30}, \eqref{Y1} and \eqref{Y2} that
\beq{eq1eq1}
\wPr(R_1\geq \r )\leq 2(1-\xi)^{3\r\ell^*}.
\eeq
Assume next that for $a_k=(2.5)^k,\,k=o(\ln M)$ and for integer $1\leq \r\leq M^{o(1)}$,
\beq{eq4}
\wPr(R_k\geq \r)\leq a_k(1-\xi)^{2\r\ell^*}
\eeq
for $t\geq 1$. This is true for $k=1$ and $a_1=2+o(1)$.
Using \eqref{eq30} and arguing as in Section \ref{Rv0} we get
\begin{align}
\wPr(R_{k+1}\geq \r)
&\leq 2\brac{\sum_{s=1}^{2\r\ell^*-1}\wPr(\ell_1=s)\wPr(R_k\geq 2\r-s)+
\wPr(\ell_1\geq 2\r\ell^*)}\label{nn1}\\
&\leq 2\brac{\sum_{s=1}^{2\r\ell^*-1}\xi(1-\xi)^{s\ell^*-1}
\times a_k(1-\xi)^{2(2\r-s)\ell^*}+(1-\xi)^{2\r\ell^*}}\nonumber\\
&=2\brac{(1+o(1))a_k\xi(1-\xi)^{4\r\ell^*}
\sum_{s=1}^{2\r\ell^*-1}(1-\xi)^{-s\ell^*}+(1-\xi)^{2\r\ell^*}}\label{nn2}\\
&\leq (2+o(1))(a_k+1)(1-\xi)^{2\r\ell^*}.\nonumber\\
&\leq a_{k+1}(1-\xi)^{2\r\ell^*}.\nonumber
\end{align}
This verifies the inductive step for \eqref{eq4}
and \eqref{R1g} follows. Remember that $(1-\xi)^{2\r\ell^*}\leq e^{-2\r\ell^*\xi}=
e^{-2\r/\om}$.

For the non locally tree like case we now argue as in Section \ref{Rv0}
down to \eqref{zzxx} and obtain
\begin{align*}
\wPr(R\geq \r)&\leq 2\brac{\sum_{s=1}^{2\r\ell^*-1}(1-\xi)^{s\ell^*/2} (2.5)^k
(1-\xi)^{2(2\r-s)\ell^*}+(1-\xi)^{2\r\ell^*}}\\
&\leq2\brac{(2.5)^k(1-\xi)^{\r\ell^*}(\xi\ell^*)^{-1}+(1-\xi)^{2\r\ell^*}
}\\
&\leb \om (2.5)^{k+1}(1-\xi)^{\r\ell^*}.
\end{align*}
There is enough slack in \eqref{R1g} to absorb the $\om$ factor when $k=L_0$.

Now suppose that $\La_f$ contains vertices from $V_\s$.
When we encounter
a component $T$ of $V_\s\cap \La_f$
we replace it $\cN$ by edges $(v_0,v_i)$ (or $(v_0',v_i)$) and these
edges will have been given the same
resistance distribution as other edges of $\La_f$, conditioned
on being heavy.
This happens with
probability $1-o(1)$ and the net result is to replace the factor
2 in \eqref{nn1} by $2+o(1)$. This will not significantly affect the
rest of the calculation here.


\subsubsection{Estimating $R_{\x}''$}\label{frontier}

It follows from Lemma \ref{plm1} that
$$R_{\x}''\leq n^{-\d_0/6}(R_{\x}'+R_{\x}'')$$
and hence
\beq{summ}
R_{\x}''\leq n^{-\d_0/7}R_{\x}'.
\eeq


\subsubsection{Completing the proof of upper bound in Case~(c) of Theorem~\ref{th1}}\label{complete}
We are almost ready to apply \eqref{mm9}. We have estimated the cover time, but we have
ignored some of the time.
Specifically, let
$$E_1=\bigcup_{\substack{e\in E_K\\e\cap V_\s\neq\emptyset}}P_e.$$
We have not accounted for the time that $\CW{G_F}$ spends covering $E_1$.

For this we can apply a theorem of Gillman \cite{Gill}: Let $G=(V,E)$ be an edge
weighted graph and for $x\in V$ let $N_\bq=\card{\card{\frac{\bq}{\sqrt{\p}}}}_2$ where
$\p(x),x\in V$ is the steady state
distribution for the associated random walk and $q(x),x\in V$ is any initial distribution
for the starting point
of the walk.
Let $\th$ denote the spectral gap for the associated probability transition matrix.
\begin{theorem}\label{Gill}
Let $A\subseteq V$ and let $Z_t$ be the number of visits to $A$ in $t$ steps. Then, for
any $\g\geq 0$,
$$\Pr(Z_t-t\p(A)\geq \g)\leq (1+\g\th/10t)N_\bq e^{-\g^2\th/20t}.$$
\end{theorem}
We apply this theorem to the random walk $\CW{G_F}$. Let $A=E_1$ and
$\g=M/\xi^2$.
It follows from Lemmas~\ref{lem2} (Part~(c)) and \ref{degVsigma} that w.h.p.
$$\p(A)=O\bfrac{\om^{-1/3}M\times \xi^{-1}\ln M}{M+\n_2}=O(\om^{-1/3}\ln M).$$
It follows from Lemma \ref{lem3} that $\th=
\Omega(\xi^2/\ln^2M)$. Now let $t=M\ln^2M/\xi^2$.
Then
with $\bq$ of the form $(0,0,\ldots,1,0,\ldots,0)$ we have
$$\Pr(Z_t\geq t\p(A)+\g)=O(m^{1/2}e^{-\Omega(M/\ln^4M)})=o(1).$$
This completes the proof of Case (c2).
\subsection{Case (c3): $\n_2\geq e^\om$}\label{casec3}
In this case we can
use the fact that w.h.p.
$\ell_e\in[\bm,\am]$ for $e\in E_K$ to (i) partition all induced paths of $G_F$ into
sub-paths of
length $\sim\m=me^{-\om/2}$, (ii) replace these sub-paths by edges to create a graph
$\G$ and then (iii) apply
the Case~(c) reasoning to $\G$ and then scale up by $\m^2$
to get the claimed upper bound.

The proof of the upper bound for Case~(c) of Theorem~\ref{th1} is now complete.

\subsection{Case~(b): $\n_2=M^\a,\,0<\a<1$}\label{caseb}
Our argument for this case will not be so detailed as for the previous cases.
It is closer in spirit to
that of the previous papers of the first two authors.

Note that in this case
$$1-\xi\leq \frac{1}{M^{1-\a}}.$$
So,
\begin{lemma}\label{largeste}
Let $\th>0$ be an arbitrarily small positive constant. Then
w.h.p.\ $\ell_e\leq \ell_\a=\rdup{1/(1-\a)+1+\th}$ for $e\in E(K_F)$.
\end{lemma}
\begin{proof}[\textbf{\emph{Proof}}]
Going back to \eqref{830} we have
\begin{equation*}
\Pr(\exists e:\ell_e\geq \ell_\a)\leq M\sum_{s\geq \ell_\a}
M^{-(1-\a)(s-1)}\brac{1+\frac{3}{M+M^\a}}^{s-1}
=o(1).\qedhere
\end{equation*}
\end{proof}

The next thing to observe in this case that there
will be very few vertices of degree two close to any vertex of $K_F$.
Suppose that $d_n=\D$. We choose $\d_0\leq 1/100$
such that $\D^{L_0}\leq M^{(1-\a)/2}$.
Let $E_{v,s}$ be the set of edges of $K_F$ that are
within distance $s$ of vertex $v\in V(K_F)$.
\begin{lemma}\label{sparse}
 W.h.p., for all $v\in V(K_F)$,
$$\sum_{e\in E_{v,L_0}}\ell_e\leq |E_{v,L_0}|+2\ell_\a.$$
\end{lemma}
\begin{proof}[\textbf{\emph{Proof}}]
Let $ h_v=|E_{v,L_0}|\leq 2M^{(1-\a)/2}$.
Then we have
\begin{align*}
\Pr\brac{(\sum_{e\in E_{v,L_0}}\ell_e\geq h_v+2\ell_\a}&
\leq o(1)+\sum_{v\in V(K_F)}\sum_{s\geq h_v+2\ell_\a}
\sum_{\substack{z_e,e\in E_{v,L_0}
\\ \sum_ez_e=s}}
M^{-(1-\a)(s- h_v)}\brac{1+\frac{3}{M+M^\a}}^{s- h_v}\\
&\leq o(1)+M\sum_{s\geq  h_v+2\ell_\a}\binom{s-1}{ h_v-1}M^{-(s- h_v)(1-\a+o(1))}\\
&\leb o(1)+ M\sum_{s\geq  h_v+2\ell_\a}\brac{\frac{se}{s- h_v}\cdot \frac{1}{M^{1-\a+o(1)}}}
^{s- h_v}\\
& \leq o(1)+ M\sum_{s\geq h_v+2\ell_\a}\brac{\frac{eh_v}{2\ell_\a}\cdot \frac{1}{M^{1-\a+o(1)}}}^{s-h_v}\\
& \leq o(1)+ M\sum_{s\geq h_v+2\ell_\a}M^{-(s-h_v)(1-\a+o(1))/2}\\
&=o(1).\qedhere
\end{align*}
\end{proof}

It is not difficult to show that the conditions of Lemma \ref{MainLemma} hold w.h.p.
and so it is a matter of estimating
the $R_v$'s. This involves estimating the effective resistances $R'_v$ so that
we can use \eqref{resist0}. The inequalities
\begin{align*}
&1+\frac{1}{{\frac{1}{R-1}+\frac{1}{S}}}\geq \frac{1}{{\frac{1}{R}+\frac{1}{S}}}\\
&\frac{1}{R+1}+\frac{1}{S-1}\leq \frac{1}{R}+\frac{1}{S}\text{ for positive integers }R<S
\end{align*}
imply the following:
\begin{enumerate}[(i)]
\item If $v\in V_K$ and if we assume $k=O(1)$ vertices of degree
two within distance $L_0$ of $v$ then
we get the maximum effective resistance in \eqref{resist0}
by distributing these degree two vertices equitably on the edges incident
with $v$.
\item If $\deg(v)=2$ then we get the maximum resistance when $v$ is in the
middle of the path $P_e$ that it lies.
\end{enumerate}
There are now three cases to consider:
\begin{enumerate}[(1)]
\item If $k=0$ and $v$ is locally tree like, then the resistance satisfies
\beq{C1}
R_v'\leq \r_d=\frac{d-1}{d(d-2)},
\eeq
where $d$ is the minimum degree in
$K_F$. The value $\frac{d-1}{d(d-2)}$ is the resistance $R_{d,\infty}$ of an infinite
$d$-regular tree $T_\infty$.
Trimming the tree at depth $L_0$
explains the inequality. We obtain the resistance of $T_\infty$ by first computing the
resistance $\r$ of an
infinite tree with
branching factor $d-1$. This satisfies the recurrence $\frac{1}{\r}=\frac{d-1}{1+\r}$
giving $\r=\frac{1}{d-2}$.
The resistance $R_{d,\infty}$ then satisfies
$\frac{1}{R_{d,\infty}}=\frac{d}{1+\r}$, giving
$R_{d,\infty}=(1+\r)/d$.

If on the other hand,
$k=pd+q$ where $0\leq q<d$ then
\begin{align*}
\frac1{R_v'}&\ge \brac{\frac{d-q}{p+\frac{1}{d-2}}+\frac{q}{p+1+\frac{1}{d-2}}}\\
&=\frac{d}{p+1+\frac{1}{d-2}}+\frac{d-q}{\brac{p+
\frac{1}{d-2}}\brac{p+1+\frac{1}{d-2}}}\\
&=\frac{d}{\frac{k}{d}+\frac{1}{d-2}}+
\frac{d-q}{\brac{p+\frac{1}{d-2}}\brac{p+1+\frac{1}{d-2}}}
-\frac{d-q}{\brac{\frac{k}{d}+\frac{1}{d-2}}\brac{p+1+\frac{1}{d-2}}}\\
&\geq \frac{d}{\frac{k}{d}+\frac{1}{d-2}}.
\end{align*}
The case \eqref{C1} is equivalent to $p=q=0$.

Next observe that the number of vertices with this value of
$k$ is $O(M^{1-(1-\a)k})$ w.h.p.
Thus the main contribution from these vertices to $\Psi(V,t)$ can be bounded by
\bem{C1b}
\leb \sum_{s\geq t}M\exp\set{-(1+o(1))\frac{d}{2M}\cdot\frac{s}{d \r}}+\\
\sum_{s\geq t}\sum_{k\geq 1}M^{1-(1-\a)k}
\exp\set{-(1+o(1))\frac{s}{2M}\cdot\frac{d}{\frac{k}{d}+\frac{1}{d-2}}}
\end{multline}

\item If $v\in P_e$, $e$ is locally tree like and $v$ is the
middle of $k\geq 1$ vertices of degree two, then
\beq{C2}
\frac{1}{R_v'}\geq \brac{\frac{1}{\rdown{(k+1)/2}+\frac{1}{d-2}}+
\frac{1}{\rdup{(k+1)/2}+\frac{1}{d-2}}}.
\eeq
Observe that once again the number of vertices
with this value $k$ is $O(M^{1-(1-\a)k})$ w.h.p.
Thus the main contribution from these vertices to $\Psi(V,t)$ can be bounded by
\beq{C2a}
\leb \sum_{s\geq t}\sum_{k\geq 1}M^{1-(1-\a)k}\exp\set{-(1+o(1))\frac{s}{2M}
\brac{\frac{1}{\rdown{(k+1)/2}+\frac{1}{d-2}}+\frac{1}{\rdup{(k+1)/2}+\frac{1}{d-2}}}}
\eeq
Comparing \eqref{C1b} and \eqref{C2a} we see that the latter dominates,
 except possibly for the first term corresponding
to \eqref{C1}. As in \cite{ACF},
this first term forces $\Tc\geq (1+o(1))\frac{2\r_d}{d}M\ln M$.
The other terms in \eqref{C1b} force
$$\min_k\set{(1-\a)k\ln M+\frac{\Tc}{2M}\brac{\frac{1}{\rdown{(k+1)/2}+
\frac{1}{d-2}}+\frac{1}{\rdup{(k+1)/2}+\frac{1}{d-2}}}}
\geq (1+o(1))\ln M.$$

\item Non locally tree like edges and vertices:
This follows from two easily proven facts: (i) There are $M^{o(1)}$ such vertices
and edges, (ii) the resistance $R_v'$ in all such cases is $O(1/(1-\a))$. This means
that all such vertices will w.h.p.\ have been visited after $o(M\ln M)$ steps.
\end{enumerate}

This completes the upper bound for Case~(b) of Theorem~\ref{th1}.
\subsection{Case~(a): $\n_2=M^{o(1)}$}\label{casea}
This is essentially treated in \cite{ACF}. W.h.p.\ every $K_F$ neighbourhood up to depth
$L_0$ attracts at most one vertex of degree two when edges are split. Furthermore all
but an $M^{-(1-o(1)}$ fraction are free of vertices of degree two. It is easy therefore
to amend the proof in \cite{ACF} to handle this.
\section{Lower Bounds}\label{lowerbounds}
\subsection{Case~(a): $\n_2=M^{o(1)}$}\label{lowera}
This is essentially treated in \cite{ACF}.
\subsection{Case~(b): $\n_2=M^\a,\,0<\a<1$}\label{lowerb}
This can be treated via the second moment method as described in \cite{CFreg}.
We give a bare outline of the approach. Let
\[ \psi_{\a,d} = \max\left\{\frac{2(d-1)}{d(d-2)}\,, \phi_{a,d}\right\}\,,\]
set $t=(1-o(1))\psi_{\a,d}M\ln M$
and suppose for example that $\psi_{\a,d}=\frac{2(d-1)}{d(d-2)}$. This is true
for $\a$ small and $d$ large. We then let $S$ denote the set of vertices that
(i) are locally tree like, (ii) have no degree
two vertices added to their $L_0$-neighbourhood
and (iii) have only degree $d$ vertices in their $L_0$-neighbourhood.
We find that $|S|=\Omega(n^{1-o(1)})$ w.h.p.\ and we greedily choose a sub-set $S_1$ of $S$
so that (i) if $v,w\in S_1$ then $\mathrm{dist}(v,w)>2L_0$ and (ii) $|S_1|=n^{1-o(1)}$.
Let $S^*$ denote the set of vertices in $S_1$ that remain unvisited at time $t$.
We choose the $o(1)$ term in the definition of $t$ so that $\E(|S^*|)\to\infty$.
We will then argue that if $v,w\in S_1$ then
\beq{nearend}
\Pr(\cA_t(v)\cap \cA_t(w))\sim \Pr(\cA_t(v))\Pr(\cA_t(w)).
\eeq
This means, via the Chebyshev inequality,
that w.h.p.\ $S^*\neq \emptyset$, giving the
lower bound. To prove \eqref{nearend} we
consider a new graph $G'$ where we identify
$v,w$ to make a vertex $\Upsilon$ of degree
$2d$. We then apply Lemma \ref{MainLemma} to $G'$
to estimate $\Pr(\cA_t(\Upsilon))$. Observe
that up until the walk visits $\Upsilon$ in $G'$,
its moved can be coupled with moves in $G$.
Also, $\upsilon$ has steady state probability
approximately equal to that of $v,w$ combined, but $R_\Upsilon\sim R_v\sim R_w$ and
\eqref{nearend} follows.

\subsection{Case~(c): $\n_2=\Omega(M^{1-o(1)})$}\label{lowerc}
We use the following result of Matthews~\cite{Mat}. For any graph $G$
\[
\Tc(G) \ge \frac{1}{2} \max_{S \subset V} K_S \ln |S|,
\]
where
\[
K_S = \min_{u,v \in S} K(u,v).
\]
Here $K(u,v)$ is commute time between $u$ and $v$, i.e., the expected time
for a walk $\cW$ that starts at $u$ to visit $v$ and then return to $u$.
This in turn is given by
\[
K(u,v)=2 |E(G)| R_{\mathrm{eff}}(u,v),
\]
$E(G)$ is edges of $G$, and $R_{\mathrm{eff}}(u,v)$ is effective resistance between $u$ and $v$.

It is now simply a matter of finding a suitable set $S$.

Fix an integer $\ell$ and consider
$$S_\ell=\set{u:\exists e\in \tKer\text{ such that }u\text{ is the middle vertex of }P_e
\text{ and }\ell_e\geq \ell}.$$
Now
$$R_{\mathrm{eff}}(u,v)\geq \ell/2\text{ for }u,v\in S_\ell.$$
To see this, let $P_e$, $P_f$ be two paths of length (at least)
$\ell$ and let $a,b,c,d$ be their respective endpoints. Let $u,v$
be the midpoints of $P_e,P_f$. Let $V_{\l}$ be the set of vertices not on
$P_e$ or $P_f$. Contract the set $V_{\l}\cup \{a,b,c,d\}$ to a single
vertex $z$. This does not increase the effective
resistance between $u$ and $v$. What results is a
graph consisting of two cycles intersecting at $z$. The effective
resistance between $u$ and $v$ is now
at least $\ell/4+\ell/4=\ell/2$. Here $\ell/4$
is a lower bound on the resistance between $u$ and $z$ etc.

Now $m\geq \n_2$ and we will choose our $\ell$ to be $\ell_0=\frac{\ln M}{-2\ln(1-\xi)}$.
It follows from Lemma \ref{lem2} (Part~(b)) with $k=1$ that
$\E(|S_{\ell_0}|)\sim M(1-\xi)^{\ell_0}$.
Lemma \ref{lem2} (Part~(b)) with $k=2$ allows us to use the Chebyshev inequality
to show that $|S_{\ell_0}|\sim M(1-\xi)^{\ell_0}$ w.h.p.\ (Here we take $\z\leq 2\am$ so that
$\frac{\z}{\n_2+M}=O\bfrac{\ln^2M}{M}=o(1)$.)
Note that $M(1-\xi)^{\ell_0}=M^{1/2}\to\infty$.

Putting this altogether we see that w.h.p.
\beq{mm1}
\Tc(\tTC) \ge (1-o(1))\n_2\times \frac{\ln M}{-4\ln(1-\xi)}
\times \frac{\ln M}{2}.
\eeq
Since $-\ln(1-\xi)\sim\xi$ for small $\xi$, this also includes Case~(c).
This completes the proof of Case~(c) of Theorem~\ref{th1}.
\begin{Remark}
Our assumption, $-\ln(1-\xi)=o(\ln M)$ implies that
we can ignore the fact that $\ell_0$ is an integer.
That is, by
defining $\ell_0$ without $\rdup{\cdot}$ we can include the error in the $(1-o(1))$ factor.
\end{Remark}

\bigskip
{\bf Acknowledgment.}\, This work was initiated when A.F.\ was
visiting Microsoft Research, Redmond, and he would like to thank
the Theory Group at Microsoft Research for its hospitality and for
creating a stimulating research environment.
The authors are also grateful to the anonymous referees for useful comments and corrections.

\end{document}